\documentclass{amsart}

\usepackage{amssymb}
\usepackage{amsrefs}

\usepackage{amsfonts}
\usepackage{mathrsfs}
\usepackage{bbm}

\usepackage[dvipdfmx]{graphicx}

\theoremstyle{plain}
    \newtheorem{theorem}{Theorem}[section]
    \newtheorem{lemma}[theorem]{Lemma}
    \newtheorem{proposition}[theorem]{Proposition}
    \newtheorem{corollary}[theorem]{Corollary}
\theoremstyle{definition}
    
    \newtheorem{remark}{Remark}[section]
    
\theoremstyle{remark}
    
\numberwithin{equation}{section}

\renewcommand{\l}{\left}
\renewcommand{\r}{\right}
\newcommand{\cleq}{\lesssim}
\newcommand{\cgeq}{\gtrsim}
\newcommand{\eps}{\varepsilon}
\def\norm#1{\left\Vert #1 \right\Vert} 
\def\tbra#1#2{\left\langle #1 , #2 \right\rangle} 

\newcommand{\N}{\mathbb{N}}

\newcommand{\R}{\mathbb{R}}

\newcommand{\cD}{\mathcal{D}}
\newcommand{\cE}{\mathcal{E}}
\newcommand{\cF}{\mathcal{F}}
\newcommand{\cG}{\mathcal{G}}

\newcommand{\cI}{\mathcal{I}}

\newcommand{\cM}{\mathcal{M}}

\DeclareMathOperator{\re}{Re}

\begin{document}

\title[Asymptotic behavior of the solutions to NLS on star graph]{Asymptotic behavior for the long-range nonlinear  Schr\"odinger equation on star graph with the Kirchhoff boundary condition}
\author[K. Aoki]{Kazuki Aoki}
\address{Department of Mathematics, Graduate School of Science, Osaka University, Toyonaka, Osaka 560-0043, Japan}
\email{k.aoki5296@gmail.com}
\author[T. Inui]{Takahisa Inui}
\address{Department of Mathematics, Graduate School of Science, Osaka University, Toyonaka, Osaka 560-0043, Japan}
\email{inui@math.sci.osaka-u.ac.jp}
\author[H. Miyazakii]{Hayato Miyazaki}
\address{Advanced Science Course, Department of Integrated Science and Technology, National Institute of Technology, Tsuyama College, Tsuyama, Okayama, 708-8509, Japan}
\email{miyazaki@tsuyama.kosen-ac.jp}
\author[H. Mizutani]{Haruya Mizutani}
\address{Department of Mathematics, Graduate School of Science, Osaka University, Toyonaka, Osaka 560-0043, Japan}
\email{haruya@math.sci.osaka-u.ac.jp}
\author[K. ]{Kota Uriya}
\address{Department of Applied Mathematics, Faculty of Science, Okayama University of Science, Okayama, 700-0005, Japan}
\email{uriya@xmath.ous.ac.jp}
\date{\today}

\maketitle

\begin{abstract}
We consider the cubic nonlinear Schr\"{o}dinger equation on the star graph with the Kirchhoff boundary condition. 
We prove modified scattering for the final state problem and the initial value problem. Moreover, we also consider the failure of scattering for the Schr\"{o}dinger equation with power-type long-range nonlinearities. These results are extension of the results for NLS on the one dimensional Euclidean space. 
\end{abstract}

\tableofcontents


\section{Introduction}

\subsection{Background}
We mainly consider the following cubic nonlinear Schr\"{o}dinger equation on the star graph $\mathcal{G}$ with $n$-edges:
\begin{align}
\label{eq:NLS}
	i\partial_t u +\Delta_{K} u +\lambda |u|^2 u = 0, 
	\quad (t,x) \in I \times  \mathcal{G},
\end{align}
where $I$ is a time interval, $\lambda=\pm1$, and $\Delta_{K}$ is the Laplacian with the Kirchhoff boundary condition on the star graph $\mathcal{G}$.
Recently, researches of dispersive equations on metric graphs
have attracted much attention. 
Roughly, the star graph is a metric graph such as in Figure \ref{fig1} below.
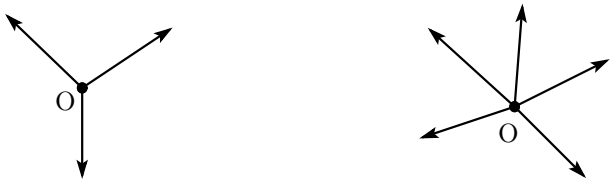
\begin{figure}[htbp]
 \begin{center}
{\unitlength 0.1in%
\begin{picture}(9.1300,8.3100)(4.1100,-12.5700)%
%
\special{pn 13}%
\special{pa 800 800}%
\special{pa 1257 495}%
\special{fp}%
\special{sh 1}%
\special{pa 1257 495}%
\special{pa 1190 515}%
\special{pa 1213 525}%
\special{pa 1213 549}%
\special{pa 1257 495}%
\special{fp}%
\special{pa 800 800}%
\special{pa 800 1257}%
\special{fp}%
\special{sh 1}%
\special{pa 800 1257}%
\special{pa 820 1190}%
\special{pa 800 1204}%
\special{pa 780 1190}%
\special{pa 800 1257}%
\special{fp}%
\special{pa 800 800}%
\special{pa 411 426}%
\special{fp}%
\special{sh 1}%
\special{pa 411 426}%
\special{pa 445 487}%
\special{pa 449 463}%
\special{pa 473 458}%
\special{pa 411 426}%
\special{fp}%
%
\special{pn 13}%
\special{ar 800 800 22 24 0.1460123 0.0822317}%
%
\special{pn 20}%
\special{pa 815 805}%
\special{pa 806 814}%
\special{fp}%
\special{pa 816 798}%
\special{pa 806 808}%
\special{fp}%
%
\special{pn 20}%
\special{pa 793 803}%
\special{pa 787 809}%
\special{fp}%
\special{pa 790 800}%
\special{pa 785 805}%
\special{fp}%
\special{pa 787 797}%
\special{pa 784 800}%
\special{fp}%
\special{pa 794 808}%
\special{pa 790 812}%
\special{fp}%
%
\special{pn 20}%
\special{pa 806 784}%
\special{pa 799 791}%
\special{fp}%
\special{pa 801 783}%
\special{pa 796 788}%
\special{fp}%
\put(6.5800,-9.1400){\makebox(0,0)[lb]{O}}%
%
\special{pn 20}%
\special{ar 800 800 11 11 6.2387701 6.2033553}%
%
\special{pn 20}%
\special{ar 800 800 18 18 0.0000000 6.2831853}%
\end{picture}}%
  \hskip30mm
{\unitlength 0.1in%
\begin{picture}(9.9200,8.8000)(2.8800,-11.6000)%
%
\special{pn 13}%
\special{ar 800 800 22 24 0.1460123 0.0822317}%
%
\special{pn 20}%
\special{pa 815 805}%
\special{pa 806 814}%
\special{fp}%
\special{pa 816 798}%
\special{pa 806 808}%
\special{fp}%
%
\special{pn 20}%
\special{pa 793 803}%
\special{pa 787 809}%
\special{fp}%
\special{pa 790 800}%
\special{pa 785 805}%
\special{fp}%
\special{pa 787 797}%
\special{pa 784 800}%
\special{fp}%
\special{pa 794 808}%
\special{pa 790 812}%
\special{fp}%
%
\special{pn 20}%
\special{pa 806 784}%
\special{pa 799 791}%
\special{fp}%
\special{pa 801 783}%
\special{pa 796 788}%
\special{fp}%
\put(7.1400,-9.8400){\makebox(0,0)[lb]{O}}%
%
\special{pn 13}%
\special{pa 800 800}%
\special{pa 320 960}%
\special{fp}%
\special{sh 1}%
\special{pa 320 960}%
\special{pa 390 958}%
\special{pa 371 943}%
\special{pa 377 920}%
\special{pa 320 960}%
\special{fp}%
%
\special{pn 13}%
\special{pa 800 800}%
\special{pa 840 280}%
\special{fp}%
\special{sh 1}%
\special{pa 840 280}%
\special{pa 815 345}%
\special{pa 836 333}%
\special{pa 855 348}%
\special{pa 840 280}%
\special{fp}%
%
\special{pn 13}%
\special{pa 800 800}%
\special{pa 360 400}%
\special{fp}%
\special{sh 1}%
\special{pa 360 400}%
\special{pa 396 460}%
\special{pa 399 436}%
\special{pa 423 430}%
\special{pa 360 400}%
\special{fp}%
\special{pa 800 800}%
\special{pa 1280 560}%
\special{fp}%
\special{sh 1}%
\special{pa 1280 560}%
\special{pa 1211 572}%
\special{pa 1232 584}%
\special{pa 1229 608}%
\special{pa 1280 560}%
\special{fp}%
%
\special{pn 13}%
\special{pa 800 800}%
\special{pa 1160 1160}%
\special{fp}%
\special{sh 1}%
\special{pa 1160 1160}%
\special{pa 1127 1099}%
\special{pa 1122 1122}%
\special{pa 1099 1127}%
\special{pa 1160 1160}%
\special{fp}%
%
\special{pn 20}%
\special{ar 800 800 14 14 0.0000000 6.2831853}%
%
\special{pn 20}%
\special{ar 800 800 18 18 0.0000000 6.2831853}%
\end{picture}}%
 \end{center}
 \caption{3-edges star graph and 5-edges star graph}
 \label{fig1}
\end{figure}
We give precise definition of the star graph and setting of the problem later. 

In the present paper, we are interested in the asymptotic behavior of the solution 
to \eqref{eq:NLS} and consider both final state problem and 
 initial value problem. 
Both problems in the one dimensional Euclidean space are
extensively studied for the nonlinear Schr\"odinger equation with 
gauge invariant power-type nonlinearity: 
\begin{equation}
\label{eq:gNLS}
i\partial_tu + \Delta u + \lambda |u|^{p}u = 0, \quad (t,x) \in \R \times \R, 
\end{equation}
where $p > 0$ and $\lambda = \pm 1$.
It is well-known that $p = 2$ is the critical exponent 
in the sense of the asymptotic behavior of the solutions. 
More precisely, if $p >2$, the solution to \eqref{eq:gNLS}
scatters to a solution of the free Schr\"odinger equation \cite{TsYa}. 
In this case, the nonlinearity is called ``short-range". 
On the other hand, if $0 < p \le 2$, the solutions to 
\eqref{eq:gNLS} do not scatter to free solutions. 
See for instance \cite{Str74,Ba,Caz,MuNa} and the references given there. 
In the critical case, i.e. $p = 2$, 
it is known that the solution of \eqref{eq:gNLS} scatters to a free solution with a phase modification.
This phenomena is called modified scattering. 	
In the case of final state problem, 
Ozawa~\cite{Oz} showed the modified scattering. 
In the case of initial value problem,
Hayashi--Naumkin~\cite{HaNa98} showed the modified scattering. 
Several alternative proofs of the result \cite{HaNa98} were given by Lindbald--Soffer~\cite{LiSo}, Kato--Pusateri \cite{KaJ}, and Ifrim--Tataru~\cite{IfTa}. Moreover, the modified scattering phenomena for the nonlinear Schr\"{o}dinger equation on the higher dimensional Euclidean space $\mathbb{R}^d$ ($d=2,3$) was also proved by e.g. Ginibre--Ozawa~\cite{GO} and Hayashi--Naumkin \cite{HaNa98}. (Note that the critical exponent is given by $p=2/d$ in this case.)

We consider the Schr\"{o}dinger equation on the star graph with power-type nonlinearity:
\begin{align}
\label{NLS^p}
i \partial_{t} u +\Delta_{K} u +\lambda \left| u \right|^{p}u=0, \quad  t \in \mathbb{R},\ x \in \mathcal{G},
\end{align}
where $\lambda=\pm 1$ and $p>0$. Since the star graph is the connected half-lines (the star graph with 2-edges is just a line), the critical exponent is expected to be $p=2$ as the case of the line $\mathbb{R}$. 
Yoshinaga \cite{Yos18} proved that the solutions of \eqref{NLS^p} scatter to the free solution when $p>2$, whose argument is based on \cite{TsYa}. (See also \cite[Remark 3]{AIM20} for the precise statement.) Recently, the first, second, and fourth authors \cite{AIM20} proved the failure of scattering when $0<p<1$ (In fact, they treated the nonlinear Schr\"{o}dinger equation with more general boundary conditions). 

In the present paper, we are mainly interested in the scattering phenomena in the critical case $p=2$. We will give modified scattering results for the final state problem and the initial value problem when $p=2$. 
Our proofs are based on the argument of \cite{Oz,HaNa98}. Namely, we use the factorization formula, which is also called the Dollard decomposition, of the propagator $e^{it\Delta_{K}}$. 
To derive the factorization formula, we apply the Fourier transform $\mathcal{F}$ with respect to $\Delta_{K}$ derived by Weder \cite{Wed15}, which is an extension of the usual Fourier transform on the line $\mathbb{R}$. 
We also have an interest in the asymptotic behavior of the solutions in the case of $1 \leq p \leq 2$.  
We will show that the scattering to the free solution fails when $1\leq p \leq 2$ by assuming more regularity than in the previous paper\cite{AIM20}.
The proof is based on \cite{Caz}.

\subsection{Setting and notations}
\label{sec2}
Before stating the main results, we give some notations used in the main results and their proofs.
A finite graph is a 4-tuple $(V, \mathcal{I}, \mathcal{E}, \partial)$, 
where $V$ is the finite set of the vertices , $\cI$ is the finite set of 
internal edges, $\cE$ is the finite set of external edges. 
A map $\partial$ is from $\cI \cup \cE$ to the set of vertices and ordered 
pairs of two vertices which satisfies $\partial(i) = (v^i_1,v^i_2)$ 
(possibly $v^i_1 = v^i_2$, $v^i_1, v^i_2 \in V$) for $i \in \cI$
and $\partial(e) = v$ for $e \in \cE$. 
We call $v^i_1 =: \partial^-(i)$ and $v^i_2 =: \partial^+(i)$
initial and final vertex of the internal edge $i \in \cI$, respectively. 
We endow the graph with the metric structure. We assume that for any internal edge $i \in I$ there exist $a_i >0$ and a map $i\mapsto [0,a_i]$ corresponding $\partial^{-}(i)$ to $0$ and $\partial^{+}(i)$ to $a_i$ and that for any external edge $e \in \mathcal{E}$ there exists a map $e \mapsto [0,\infty)$. 
We call $a_i$ the length of the internal edge $i \in \cI$. 
The graph endowed with such metric structure is called metric graph. 
For given $n \in \N$, a star graph with $n$-edges is a metric graph 
$(\{0\}, \emptyset, \{e_j\}_{j=1}^n, 
\partial : \{e_j\}_{j=1}^n\to \{0\})$. 
See Figure \ref{fig1}  for typical examples of star graph. 
Throughout the paper, let 
$\cG = (\{0\}, \emptyset, \{e_j\}_{j=1}^n, \partial : \{e_j\}_{j=1}^n\to \{0\})$ 
be a star graph. 

A function $f$ on $\cG$ is given by 
a vector $f = (f_1, f_2, \cdots, f_n)^T$, where each $f_j$ is a 
complex-valued function defined on $e_j = (0,\infty)$
and $f^T$ denotes the transpose of $f$. 
We emphasize that the notation $|f|$ for a function $f$ on $\mathcal{G}$ does not mean $(\sum_{j=1}^{n}|f_j|^2)^{1/2}$. In the paper, we regard $|f|$ as $(|f_j|)_{1\leq j \leq n}$. Moreover $fg=(f_jg_j)_{1\leq j \leq n}$ for functions $f,g$ on $\mathcal{G}$. That is, our calculation is component-wise. Especially, the nonlinearity $|f|^2f$ is $(|f_j|^2f_j)_{1\leq j \leq n}$.  

The Lebesgue measure on $\cG$ is naturally induced by the Lebesgue measure 
on half-lines $e_1, \cdots, e_n$. 
The function space $L^2(\cG)$ is defined as the set of measurable functions 
which are square-integrable on each external edge of $\cG$. 
Namely, 
\begin{equation}
L^2(\cG) = \bigoplus_{j=1}^n L^2(e_j)
\end{equation}
and the inner product and the norm are defined by 
\begin{align}
\tbra{f}{g}=\tbra{f}{g}_{L^2(\mathcal{G})} &:= \sum_{j=1}^n\tbra{f_j}{g_j}_{e_j}
= \sum_{j=1}^n\int_{e_j}f_j(x) \overline{g_j(x)}dx, \\
\|f\|_{L^2(\cG)}^2 &:= \tbra{f}{f}
= \sum_{j=1}^n\int_{e_j}|f_j(x)|^2dx = \sum_{j=1}^n\|f_j\|_{L^2(e_j)}^{2}, 
\end{align}
where $f = (f_j)_{j = 1, \cdots, n}^T$, $g = (g_j)_{j=1,\cdots, n}^T$
with  $f_j$, $g_j \in L^2(e_j)$ for each $j=1,\cdots, n$. 
Then $L^2(\cG)$ is a Hilbert space. 
For $1 \le p \le \infty$, $L^p(\cG)$ can be defined similarly, 
i.e. $f \in L^p(\cG)$ if each component of $f$ is $L^p$ function. 
The norm is defined by 
\begin{equation}
\|f\|_{L^p}=
\|f\|_{L^p(\cG)}
:= \left\{
\begin{aligned}
&\left(\sum_{j=1}^n\|f_j\|_{L^p(e_j)}^p\right)^{\frac{1}{p}}
\quad (1 \le p < \infty), \\
&\sup_{1 \le j \le n} \|f_j\|_{L^\infty(e_j)}
\quad (p = \infty).
\end{aligned}
\right.
\end{equation}
The norm of the weighted $L^2$ space, which is denoted by $H^{0,1}(\mathcal{G})$, is defined by 
\begin{align*}
	\| f \|_{H^{0,1}}^2 =\| f \|_{H^{0,1}(\mathcal{G})}^2 := \| f \|_{L^2 (\mathcal{G})}^2 + \| Xf \|_{L^2(\mathcal{G})}^2,
\end{align*}
where $Xf=(xf_j)_{j = 1, \cdots, n}^T$.

For $m=1,2$, the Sobolev space $H^m(\cG)$ is defined by 
\begin{equation}
H^m(\cG) := \bigoplus_{j=1}^nH^m(e_j) 
\end{equation}
and the norm is defined by 
\begin{equation}
\|f\|_{H^m} =
\|f\|_{H^m(\cG)} 
:= \left(\sum_{j=1}^n\|f_j\|_{H^{m}(e_j)}^2\right)^\frac{1}{2}.
\end{equation}
We set
\begin{align*}
	\Sigma=\Sigma(\mathcal{G}) :=H^1(\mathcal{G}) \cap H^{0,1}(\mathcal{G}).
\end{align*}
We remark that we do not assume any conditions at the vertex. When we assume the continuity at the vertex, we use the subscript $c$ such as $H_c^m(\cG)$ and $\Sigma_{c}(\mathcal{G})$.

We introduce the Laplacian on the star graph with the Kirchhoff boundary. 
Let
\begin{equation}
\label{eq2.7}
A = \begin{pmatrix}
1 & -1 & 0 & \cdots & 0 & 0\\
0 & 1 & -1 & \cdots & 0 & 0\\
\vdots & \vdots & \vdots & & \vdots & \vdots\\
0 & 0 & 0 & \cdots & 1 & -1 \\
0 & 0 & 0 & \cdots & 0 & 0 
\end{pmatrix}, 
\quad 
B= 
\begin{pmatrix}
0 & 0 & 0 & \cdots & 0 & 0\\
0 & 0 & 0 & \cdots & 0 & 0\\
\vdots & \vdots & \vdots & & \vdots & \vdots\\
0 & 0 & 0 & \cdots & 0 & 0 \\
1 & 1 & 1 & \cdots & 1 & 1 
\end{pmatrix}.
\end{equation}
Then, we define the Laplacian $\Delta_K$ as follows:
\begin{align}
\label{eq2.8}
\mathscr{D}(\Delta_K) &:= \{f \in H^2(\cG)\ :\ Af(0) + Bf'(0+)=0\}, \\
\label{eq2.9}
\Delta_K f &:= (f_1'', f_2'', \cdots, f_n''),
\end{align}
where $f_j'(x) = \frac{\partial f_j}{\partial x}$. This $\Delta_K$ is called the Laplacian on the star graph $\mathcal{G}$ with the Kirchhoff boundary. In this case, the condition $Af(0) + Bf'(0+) = 0$ implies that 
$f_j(0) = f_k(0)$ for all $j,k \in \{1,2,\cdots,n\}$
and $\sum_{j=1}^nf_j'(0+) = 0$. 
Since there is no external interaction at the vertex, 
the Laplacian $\Delta_K$ is regarded as the free Laplacian 
on the star graph. 
(See \cite{KoSc,AKW11,GrIg} for the definitions of other boundaries.)

The Schr\"{o}dinger propagator $U(t)=e^{it\Delta_{K}}$ can be defined as the unitary operator on $L^2(\mathcal{G})$ by the Stone theorem.

According to Weder \cite{Wed15}, we define the Fourier transform $\mathcal{F}$ with respect to $\Delta_{K}$ and its inverse $\mathcal{F}^{-1}$ by
\begin{align*}
	\mathcal{F}&:=
	(\mathcal{F}^{-}-\mathcal{F}^{+})I_{n}
	+ \frac{2}{n} \mathcal{F}^{+}J_{n},
	\\
	\mathcal{F}^{-1}&:=\mathcal{F}^*= (\mathcal{F}^{+}-\mathcal{F}^{-})I_{n} +\frac{2}{n} \mathcal{F}^{-}J_{n},
\end{align*}
where $I_{n}$ is the identity matrix, $J_{n}$ is the matrix whose all elements are $1$, 
\begin{align*}
	[\mathcal{F}^{\pm} f](\xi) := (2\pi)^{-1/2} \int_{0}^{\infty} e^{\pm ix\xi} f(x) dx
\end{align*}
and $\mathcal{F}^*$ denotes the adjoint of $\mathcal{F}$. 
The Fourier transform with respect to more general boundary is derived by Weder \cite{Wed15}. 

We define a multiplier operator $\mathcal{M}$ and a dilation operator $\mathcal{D}$ on the star graph by
\begin{align*}
	\mathcal{M}=\mathcal{M}(t):=
	M(t) I_n,
	\quad 
	\mathcal{D}=\mathcal{D}(t):=
	D(t) I_n,
\end{align*}
where 
$[M (t)\varphi](x):=e^{\frac{i|x|^2}{4t}}\varphi(x)$ 
and $[D(t) \varphi](x):=(2i t)^{-1/2}\varphi(x/{2t})$ for $t>0$.
We denote those inverse operators by
\begin{align*}
	\mathcal{M}^{-1}&=M^{-1}I_{n} =\mathcal{M}(-t), 
	\\
	\mathcal{D}^{-1}&=D^{-1}I_{n},
\end{align*}
where $[M^{-1}(t)\varphi](x):=e^{\frac{-i|x|^2}{4t}}\varphi(x)$, and $[D^{-1}(t) \varphi](x):=(2i t)^{1/2}\varphi(2tx)$ for $t>0$.

$A \cleq  B$ means that there exists a positive constant $C$ such that $A \leq CB$. Such constants may be different from line to line. 

\section{Main results}

In the paper, we only treat the positive direction in time for simplicity. 
The first result is the modified scattering for the final state problem for \eqref{eq:NLS}.

\begin{theorem}[Modified scattering for final state problem]
\label{thm:FS}
Let $1/4<\alpha<1/2$. 
There exists $\eps_0 > 0$ with the following properties: 
For any $\varphi \in H_{c}^{1}(\mathcal{G})$ satisfying 
$\|\varphi\|_{L^\infty(\cG)} < \eps_0$, 
there exists $T \in \mathbb{R}$ and a unique solution $u \in C([ T,\infty);L^2(\cG))\cap 
L^4((T,\infty);L^\infty(\cG))$ of \eqref{eq:NLS} satisfying 
\begin{align}
\label{eq:FS1}
\norm{u_j(t) - \frac{1}{(2it)^{\frac{1}{2}}} 
\varphi_j\l( \frac{x}{2t}\r)
\exp \left( \frac{i|x|^2}{4t} +i \frac{\lambda}{2} 
\left| 
\varphi_j\l( \frac{x}{2t}\r)  
\right|^2 \log t\right) 
  }_{L^2(e_j)}
\cleq t^{-\alpha}
\end{align}
for $t\geq T$ and each $j=1,2,...,n$.
\end{theorem}

The second result is the modified scattering for the initial value problem. 
We have the small data global existence. 

\begin{theorem}[Global existence for initial value problem]
\label{iv:thm1}
There exists $\varepsilon_0>0$ such that the following assertion holds: For any $0<\varepsilon \leq \varepsilon_0$ and $u_0 \in \Sigma_c (\mathcal{G})$ with $\norm{u_0}_{\Sigma} \leq \eps$, there exists a unique global solution $u \in C([0,\infty), \Sigma_c \cap L^{\infty})$ to \eqref{eq:NLS} with $u(0)=u_0$  satisfying
\begin{align}
	\norm{u(t)}_{L^{\infty}} \le C \eps (1+|t|)^{-\frac12} \label{iv:thm2}
\end{align}
for any $t \geq 0$.
\end{theorem}

And, we obtain the following modified scattering result.
\begin{theorem}[Modified scattering for initial value problem] 
\label{iv:thm3}
Let $u(t)$ be a global solution with $u(0) = u_0$ given by Theorem \ref{iv:thm1}.  Then, if $\|u_0\|_{\Sigma} \leq \varepsilon$, there exists a unique $W \in L^\infty(\mathcal{G}) \cap L^2(\mathcal{G})$ such that
\begin{align}
	\norm{(\mathcal{F}U(-t) u)_j(t)   \exp \left( -i \frac{\lambda}{2} \int_{1}^{t} |(\mathcal{F}u)_j|^{2} \frac{d\tau}{\tau}\right) -W_j   }_{L^2(e_j) \cap L^{\infty}(e_j)} \cleq \varepsilon t^{-\frac{1}{4} + \delta} \label{iv:thm4}
\end{align}
for any $t \geq 1$, each $j =1, 2,\ldots, n$, where $\delta$ is sufficiently small depending on $\eps$.
  
Moreover, it holds that there exists a real valued function $\Psi \in L^{\infty}(\mathcal{G})$ such that
\begin{align}
	\norm{\frac{\lambda}2 \int_1^t |\cF w(\tau)|^2 \frac{d \tau}{\tau} - \frac{\lambda}2 |W|^2 \log t - \Psi}_{L^{\infty}(\cG)} \cleq \varepsilon t^{-\frac{1}{4}+\delta} \log t \label{iv:thm5}
\end{align}
for all $t \geq 1$.  Furthermore, the estimate 
\begin{align}
	\begin{aligned}
	&{}\norm{(\cF U(-t)u)_j(t) - \exp \left(  i \frac{\lambda}2 |W_j|^2 \log t + i \Psi_j \right) W_j}_{L^2(e_j) \cap L^{\infty}(e_j)} \\
	&{}\cleq \eps t^{-\frac{1}{4}+\delta } \log t 
	\end{aligned}
	\label{iv:thm6}
\end{align}
is valid for any $t \geq 1$ and each $j =1, 2,\ldots, n$. Thus, the asymptotic formula
\begin{align}
	\begin{aligned}
	u(t) ={}& \frac{1}{(2it)^{\frac12}} W_+\l( \frac{x}{2t} \r) \exp \left( \frac{i|x|^2}{4t} + i \frac{\lambda}2 \l|W_+\l( \frac{x}{2t} \r)\r|^2 \log t \right)  \\
	&{}+ O\l( \eps t^{-\frac34+\delta} \log t \r)
	\end{aligned}
	\label{iv:thm7}
\end{align}
holds for $x \in \mathcal{G}$, where $W_+ = W \exp( i \Psi )$.
\end{theorem}

The last result is the failure of scattering for \eqref{NLS^p} when $1\leq p \leq 2$.

\begin{theorem}[Failure of scattering]
\label{thm2.4}
Let $1\leq p\leq 2$. If $u$ is a forward-global solution of \eqref{NLS^p} satisfying $u \in C([0,\infty):\Sigma_c (\mathcal{G}))$ and 
\begin{align*}
	\left\| e^{-it\Delta_{K}} u\left( t \right) - v_{+} \right\|_{\Sigma\left( \mathcal{G}\right)} \to 0 \quad \left( t \to \infty \right)
\end{align*}
for some $v_{+} \in \Sigma_c ( \mathcal{G} )$, then $v_{+} \equiv 0$.
\end{theorem}



\section{Preliminalies}

In this section, we introduce some notations and prepare some lemmas. We give their proofs in Appendix \ref{appA}.

First of all, we introduce  the following factorization formula of $U(t)$, which is very useful to investigate the modified scattering. 
\begin{proposition}
\label{prop4.1}
We have  $U(t)=\mathcal{MDFM}$.
\end{proposition}
This is similar to the factorization formula of the usual Schr\"{o}dinger propagator $e^{it\Delta}=MD\mathcal{F}_{\mathbb{R}}M$ on the Euclidean space, where $\mathcal{F}_{\mathbb{R}}$ denotes the usual Fourier transform on $\mathbb{R}$.

We define the co-Fourier transform $\mathcal{F}_{c}$ by 
\begin{align*}
	\mathcal{F}_{c}:= 
	(\mathcal{F}^{-} + \mathcal{F}^{+})I_{n} - \frac{2}{n} \mathcal{F}^{+}J_{n}
\end{align*}
and its inverse $\mathcal{F}_{c}^{-1}$ is given by $\mathcal{F}_{c}^{-1} =(\mathcal{F}_{c})^* = (\mathcal{F}^{+} + \mathcal{F}^{-})I_{n} - \frac{2}{n} \mathcal{F}^{-}J_{n}$.  
Then, we have the following relation between $\mathcal{F}$ and $\mathcal{F}_{c}$.

\begin{lemma}
\label{lem4.2}
If $\varphi \in H_c^1(\mathcal{G})$, then we have
\begin{align*}
	X\mathcal{F}^{-1}\varphi &= i\mathcal{F}_{c}^{-1} \partial_x \varphi,
	\\
	\mathcal{F}_{c} \partial_x \varphi &= i X\mathcal{F}\varphi.
\end{align*}
\end{lemma}

We note that the above lemma does not hold for general $\varphi \in H^1(\mathcal{G})$. Namely, we need to assume the continuity at the origin. On the other hand,  in the following lemma, we do not need to assume the continuity. 

\begin{lemma}
\label{lem4.4}
For $\varphi \in H^{0,1}(\mathcal{G})$, we have
\begin{align*}
	\partial_{x} (\mathcal{F} \varphi)
	=- i  \mathcal{F}_{c}( X\varphi).
\end{align*}
\end{lemma}

\begin{remark}
It is worth emphasizing that $\partial_x \mathcal{F} \neq i \mathcal{F} X$ unlike the usual Fourier transform on $\mathbb{R}$. 
\end{remark}

It is known by \cite{Wed15} that the Fourier transform $\mathcal{F}$ and the co-Fourier transform $ \mathcal{F}_{c}$ are  unitary operators on $L^2(\mathcal{G})$. Namely, we have the following. 

\begin{lemma}
\label{lem4.1}
We have $\tbra{\mathcal{F}f}{g}_{L^2(\mathcal{G})} = \tbra{f}{\mathcal{F}^{-1}g}_{L^2(\mathcal{G})}$. Especially, 
$\| \mathcal{F}f\|_{L^2(\mathcal{G})}=\|f\|_{L^2(\mathcal{G})}$. Moreover, similar statements hold for the co-Fourier transform $\mathcal{F}_{c}$. 
\end{lemma}


We have the Hausdorff--Young inequality for $\mathcal{F}$ and $\mathcal{F}_{c}$ as follows.

\begin{lemma}[Hausdorff--Young inequality]
\label{HY}
For $2\leq p \leq \infty$, it is valid that 
\begin{align*}
	&\| \mathcal{F} f \|_{L^{p}(\mathcal{G})} \cleq \| f \|_{L^{p'}(\mathcal{G})},
	\\
	&\| \mathcal{F}_{c} f \|_{L^{p}(\mathcal{G})} \cleq \| f \|_{L^{p'}(\mathcal{G})},
\end{align*}
where $p'$ is the H\"{o}lder conjugate of $p$. 
\end{lemma}

We have the following decay estimate for $\mathcal{M}-1$. 

\begin{lemma}
\label{lem4.5}
Let $1\leq p \leq \infty$ and  $\alpha \in [0,1/2]$. For $t>0$, we have
\begin{align*}
	\norm{(\mathcal{M}-1)f}_{L^p(\mathcal{G})} \cleq  |t|^{-\alpha} \norm{X^{2\alpha}f}_{L^p(\mathcal{G})},
\end{align*}
where $\mathcal{M}-1$ means $\mathcal{M}-I_{n}$ and $X^{2\alpha}:= x^{2\alpha}I_n$ for $x\geq 0$. 
\end{lemma}

\begin{lemma} \label{iv:0}
Let $f \in \Sigma_{c}(\mathcal{G})$. 
There exists a constant $c_0>0$ such that 
\begin{align*}
	\norm{XU(-t)|f|^2f}_{L^2}
	& \leq c_0 \norm{f}_{L^\infty}^2 \norm{XU(-t)f}_{L^2},
	\\
	\norm{\partial_x (|f|^2f)}_{L^2} 
	& \leq c_0 \norm{f}_{L^\infty}^2 \norm{\partial_x f}_{L^2}.
\end{align*}
\end{lemma}

The Sobolev inequality holds on the star graph without the continuity at the origin. 
\begin{lemma}[Sobolev embedding]
\label{sob:1}
Let $f \in H^1(\cG)$. We have
\begin{align*}
	\| f \|_{L^{p} (\mathcal{G})} \cleq  \| f \|_{H^1(\mathcal{G})}
\end{align*}
for $2 \leq p \leq \infty$.
\end{lemma}

By \cite{GrIg}, we have the dispersive estimates and the Strichartz estimates of the propagator $U(t)$. 
We say that $(q,r)$ is an admissible pair if it satisfies $2\leq q,r\leq \infty$ and 
\begin{align*}
	\frac{1}{q} = \frac{1}{2} \left(\frac{1}{2} - \frac{1}{r}\right).
\end{align*}

\begin{lemma}[\cite{GrIg}]
\label{lem4.8}
The following dispersive estimate holds.
\begin{align*}
	\| U(t) f \|_{L^{p'}(\mathcal{G})} \cleq |t|^{\frac{1}{2} -\frac{1}{p}} \| f \|_{L^p(\mathcal{G})},
\end{align*}
where $p \in [1,2]$ and $p'$ is the H\"{o}lder conjugate of $p$.  
Moreover, we have the Strichartz estimates:
\begin{align*}
	&\| U(t) f \|_{L^q(\mathbb{R};L^r(\mathcal{G}))} \cleq \| f \|_{L^2(\mathcal{G})},
	\\
	&\left\| \int_{0}^{t} U(t-s)F(s)ds  \right\|_{L^q(\mathbb{R};L^r(\mathcal{G}))} \cleq \| F \|_{L^{\tilde{q}'}(\mathbb{R};L^{\tilde{r}'}(\mathcal{G}))},
\end{align*}
where $(q,r)$ and $(\tilde{q},\tilde{r})$ are admissible pairs and $q,r,\tilde{q}',\tilde{r}'$ are the H\"{o}lder conjugate of $q,r,\tilde{q},\tilde{r}$, respectively. 
\end{lemma}

\section{Proof of the Main results}

\subsection{Final state problem}
In this section, we consider the final state problem. Namely, we will show Theorem \ref{thm:FS}. 
We define the function space $\mathscr{X}_{\rho}$ by 
\begin{align*}
	\mathscr{X}_{\rho}:=\{ f\in C([T,\infty);L^2(\mathcal{G})): \norm{f}_{\mathscr{X}}<\rho\}
\end{align*}
for $\rho >0$, where the norm is defined by 
\begin{align*}
	\norm{f}_{\mathscr{X}}&=\sup_{t \in [T,\infty)} t^\alpha \norm{f(t)}_{\mathscr{Y}(t)},
	\\
	\norm{f(t)}_{\mathscr{Y}(t)}&=\norm{f(t)}_{L^2(\mathcal{G})} 
	+ \left(\int_{t}^{\infty} \norm{f(s)}_{L^\infty(\mathcal{G})}^4 ds\right)^{\frac{1}{4}},
\end{align*}
where $1/4 < \alpha < 1/2$. 
We define a function $w$ on $\mathcal{G}$ for a given final data $\varphi \in H^{0,1}(\mathcal{G})$ by 
\begin{align*}
	(w(t))_j=\varphi_j \exp\left(i \frac{\lambda}{2} |\varphi_j|^2 \log t\right) \text{ for } j=1,...,n
\end{align*}
and we set $u_{ap}(t):=\mathcal{MD}w(t)$.
We will find the solution of the integral equation
\begin{align}
\label{eqIE}
	u&=u_{ap}+ i\int_{t}^{\infty} U(t-\tau)\{ N(u) -  N(u_{ap})\} d\tau 
	\\ \notag
	&\qquad -i\int_{t}^{\infty} (2\tau)^{-1}U(t-\tau)R_{N(w)} d\tau 
	  + R_{w}=:\Phi_{u_{ap}}(u)=\Phi(u). 
\end{align}
where $N(u):=-\lambda |u|^2u$ and $R_{f}:=\mathcal{MDF}(\mathcal{M}-1)\mathcal{F}^{-1} f$ for a function $f$. 
To find the solution, 
it is enough to show the functional $\Phi$ is a contraction mapping on $\mathscr{X}_{\rho}$. 

\begin{remark}
Before starting the contraction argument, we give a rough sketch of the derivation of the functional \eqref{eqIE}. 
Now, by $\Delta_K U(t)= U(t) \Delta_K$ and the differential equation \eqref{eq:NLS}, we have 
\begin{align}
\label{eq2.1}
	i\partial_t (\mathcal{F}U(-t)u(t))
	&=\mathcal{F}U(-t)N(u).
\end{align}
By the definition of $w$, we also have 
\begin{align}
\label{eq2.2}
	i \partial_t w(t) = (2t)^{-1}N(w),
\end{align}
where we note that the $j$-th component of the nonlinearity $N$ is $(N(w))_{j}=-\lambda |w_j|^2 w_j$. 
By \eqref{eq2.1}, \eqref{eq2.2} and the factorization formula $U(t)=\mathcal{MDFM}$, we have 
\begin{align}
\label{eq2.4}
	i \partial_t (\mathcal{F}U(-t) u(t) -w(t))
	&=\mathcal{F}U(-t)\{ N(u) - \mathcal{MD} (2t)^{-1} N(w)\} 
	\\ \notag
	& \qquad -  (2t)^{-1}\mathcal{F}U(-t)R_{N(w)}.
\end{align}
Now, by the vector formulation, we have $(\mathcal{MD}  (2t)^{-1} N(w))_j=  (2t)^{-1} MD (N(w))_j$ and, by the gauge invariance, we obtain $ (2t)^{-1}MD|w_j|^2w_j = |MDw_j|^2MDw_j$.
Therefore, we have
\begin{align}
\label{eq5.5}
	\mathcal{MD}  (2t)^{-1} N(w) = N(\mathcal{MD}w).
\end{align}
Moreover, by the factorization formula, we also have
\begin{align}
\label{eq5.6}
	\mathcal{F}U(-t) u(t) -w(t) = \mathcal{F}U(-t) \{ u(t) - \mathcal{MD}w(t)\} -  \mathcal{F}U(-t) R_{w},
\end{align}
Combining \eqref{eq2.4}, \eqref{eq5.5}, and \eqref{eq5.6}, we obtain
\begin{align}
\label{eq5.6.0}
	i \partial_t \{\mathcal{F}U(-t) (u(t) -\mathcal{MD}w(t))\}
	&=\mathcal{F}U(-t)\{ N(u) -  N(\mathcal{MD}w)\}
	\\ \notag
	&\quad  -  (2t)^{-1}\mathcal{F}U(-t)R_{N(w)}  +i  \partial_t (\mathcal{F}U(-t) R_{w}).
\end{align}
We may assume that $u(t) -u_{ap}(t)$ will be $0$ at infinite time from the final state condition and $\mathcal{F}U(-t) R_{w}$ is a remainder term from Lemma \ref{lem4.5}. 
Therefore, integrating \eqref{eq5.6.0} on $[t,\infty)$ and recalling $u_{ap}=\mathcal{MD}w$, we obtain
\begin{align}
\label{eq2.5}
	\mathcal{F}U(-t) (u(t) -u_{ap}(t))
	&=i\int_{t}^{\infty} \mathcal{F}U(-\tau)\{ N(u) -  N(u_{ap})\} d\tau 
	\\ \notag
	&\quad  -i\int_{t}^{\infty} (2\tau)^{-1}\mathcal{F}U(-\tau)R_{N(w)} d\tau - \mathcal{F}U(-t) R_{w}.
\end{align}
Acting $U(t)\mathcal{F}^{-1}$ from the left, we obtain the integral equation $v=\Phi(v)$. 
\end{remark}

We show that $\Phi(v)-u_{ap} \in \mathscr{X}_{\rho}$ provided that $v-u_{ap} \in \mathscr{X}_{\rho}$. 
We set 
\begin{align*}
	K_1&:=\int_{t}^{\infty} U(t-\tau)\{ N(v) -  N(u_{ap})\} d\tau,
	\\
	K_2&:=\int_{t}^{\infty} (2\tau)^{-1}U(t-\tau)R_{N(w)} d\tau.
\end{align*}
By the triangle inequality, it is sufficient to estimate the $\mathscr{X}$-norms of $K_1$, $K_2$, and $R_{w}$. 
First, we estimate $K_1$. 
Now, the difference of the nonlinearity can be written by $N(v) -  N(u_{ap})=N_1(v,u_{ap})+N_2(v,u_{ap})$ such that $|N_1(v,u_{ap})| \cleq |v-u_{ap}||u_{ap}|^2$ and $|N_2(v,u_{ap})| \cleq |v-u_{ap}|^3$. 
Therefore, by the Strichartz estimate (see Lemma \ref{lem4.8}), we obtain
\begin{align*}
	\norm{K_1}_{\mathscr{Y}(t)} 
	&\cleq \norm{\int_{t}^{\infty} U(t-\tau) N_1(v,u_{ap}) d\tau}_{\mathscr{Y}(t)} 
	+  \norm{\int_{t}^{\infty} U(t-\tau) N_2(v,u_{ap}) d\tau}_{\mathscr{Y}(t)}
	\\
	&\cleq \norm{ N_1(v,u_{ap})}_{L^1(t,\infty: L^2(\mathcal{G}))} 
	+ \norm{ N_2(v,u_{ap})}_{L^\frac{4}{3}(t,\infty: L^1(\mathcal{G}))}
	\\
	&\cleq \norm{ |v-u_{ap}||u_{ap}|^2}_{L^1(t,\infty: L^2(\mathcal{G}))} + \norm{|v-u_{ap}|^3}_{L^\frac{4}{3}(t,\infty: L^1(\mathcal{G}))}.
\end{align*}
It holds from the H\"{o}lder inequality, $v-u_{ap} \in \mathscr{X}_{\rho}$, and $\| u_{ap} \|_{L^\infty} = \| \mathcal{D}w \|_{L^\infty} \cleq t^{-1/2} \| w \|_{L^\infty}$ that
\begin{align*}
\norm{ |v-u_{ap}||u_{ap}|^2}_{L^1(t,\infty: L^2(\mathcal{G}))}
&\cleq \int_t^\infty\|v(\tau)-u_{ap}(\tau)\|_{L^2(\cG)}
\|u_{ap}(\tau)\|_{L^\infty(\cG)}^2d\tau\\
&\cleq \rho\int_t^\infty \tau^{-1-\alpha}\|w\|_{L^\infty(\cG)}^2d\tau\\
&\cleq \rho\|\varphi\|_{L^\infty(\cG)}^2t^{-\alpha}.
\end{align*}
and 
\begin{align*}
&\||v-u_{ap}|^3\|_{L^\frac{4}{3}(t,\infty:L^1(\cG))}
\\
&\cleq \left(\int_t^\infty 
\|v(\tau)-u_{ap}(\tau)\|_{L^\infty(\cG)}^{\frac{4}{3}}
\|v(\tau)-u_{ap}(\tau)\|_{L^2(\cG)}^\frac{8}{3}\right)^\frac{3}{4}\\
&\cleq \left(\int_t^\infty\|v(\tau)-u_{ap}(\tau)\|_{L^\infty(\cG)}^4d\tau\right)^\frac{1}{4}
\left(\int_t^\infty\|v(\tau)-u_{ap}(\tau)\|_{L^2(\cG)}^4d\tau\right)^\frac{1}{2}\\
&\cleq \rho t^{-\alpha}\left(\int_t^\infty \rho^4\tau^{-4\alpha} d \tau \right)^{\frac{1}{2}}\\
&\cleq \rho^{3}t^{ \frac{1}{2}-3\alpha }.
\end{align*}
Thus we see that 
\begin{equation}
\label{eq:5.8}	
\|K_1\|_{\mathscr{Y}(t)} 
\cleq \rho\|\varphi\|_{L^\infty(\cG)}^2t^{-\alpha}  + \rho^{3}\tau^{ \frac{1}{2}-3\alpha }.
\end{equation}
Next, we estimate $R_{w}$. By Lemmas \ref{lem4.5} and \ref{lem4.2}, we have
\begin{align*}
	\norm{R_{w}}_{L^2} 
	&=  \norm{(\mathcal{M}-1)\mathcal{F}^{-1} w(t)}_{L^2}
	\\
	&\cleq t^{-\frac{1}{2}}  \norm{X\mathcal{F}^{-1} w(t)}_{L^2}
	\\
	&\cleq t^{-\frac{1}{2}}  \norm{\mathcal{F}_{c}^{-1} \partial_x w(t)}_{L^2}
	\\
	&\cleq t^{-\frac{1}{2}}  \norm{\partial_x w(t)}_{L^2}
	\\
	&\cleq t^{-\frac{1}{2}}   \norm{\partial_x \varphi}_{L^2}(1 + \norm{\varphi}_{L^\infty} \log t ).
\end{align*}
It holds from Lemmas \ref{lem4.2} and  \ref{lem4.5} that
\begin{align*}
	\norm{R_{w}}_{L^\infty} 
	&\cleq t^{-\frac{1}{2}} \norm{(\mathcal{M}-1)\mathcal{F}^{-1} w(t)}_{L^1}
	\\
	&\cleq t^{-\frac{1}{2}} t^{-\frac{1}{4}+\varepsilon} \norm{X^{\frac{1}{2}-2\varepsilon}\mathcal{F}^{-1} w(t)}_{L^1}
	\\
	&\cleq t^{-\frac{3}{4} +\varepsilon} \norm{(1+X)\mathcal{F}^{-1} w(t)}_{L^2}
	\\
	&\cleq t^{-\frac{3}{4} +\varepsilon} (\norm{w(t)}_{L^2}+ \norm{\mathcal{F}_{c}^{-1}\partial_x  w(t)}_{L^2})
	\\
	&\cleq t^{-\frac{3}{4} +\varepsilon} (\norm{\varphi}_{L^2}+ \norm{\partial_x  w(t)}_{L^2}),
\end{align*}
where $\varepsilon>0$ is sufficiently small.  
Thus, we get 
\begin{equation}
	\label{eq:5.9}	
\begin{aligned}
	\left( \int_{t}^{\infty} \norm{R_{w}}_{L^\infty}^4 ds  \right)^{1/4}
	&\cleq  \left(\int_{t}^{\infty} s^{-2-1+4\varepsilon}   (\norm{\varphi}_{L^2}+ \norm{\partial_x  w(s)}_{L^2})^4 ds \right)^{1/4} 
	\\
	&\cleq  t^{-\frac{1}{2}+\varepsilon} \{ \norm{\varphi}_{L^2}+ \norm{\partial_x \varphi}_{L^2}(1 + \norm{\varphi}_{L^\infty} \log t )\}.
\end{aligned}
\end{equation}
At last, we estimate $K_2$. By the Strichartz estimates, we get
\begin{equation*}
\begin{aligned}
	\norm{\int_{t}^{\infty} (2\tau)^{-1}U(t-\tau)R_{N(w)} d\tau }_{\mathscr{Y}(t)}
	\cleq \int_{t}^{\infty} \tau^{-1} \norm{R_{N(w)} }_{L^2} d\tau.
\end{aligned}
\end{equation*}
Now, it holds from  Lemmas \ref{lem4.2} and  \ref{lem4.5} that 
\begin{equation}
	\label{eq:5.10}
\begin{aligned}
	\int_{t}^{\infty} \tau^{-1} \norm{R_{N(w)} }_{L^2} d\tau 
	&= \int_{t}^{\infty} \tau^{-1} \norm{(\mathcal{M}-1)\mathcal{F}^{-1} N(w) }_{L^2} d\tau 
	\\
	&\cleq \int_{t}^{\infty} \tau^{-1-\frac{1}{2}}  \norm{X\mathcal{F}^{-1} N(w) }_{L^2} d\tau 
	\\
	&\cleq  \int_{t}^{\infty} \tau^{-\frac{3}{2}}  \norm{\mathcal{F}_{c}^{-1} \partial_x N(w) }_{L^2} d\tau 
	\\
	&\cleq  \int_{t}^{\infty} \tau^{-\frac{3}{2}}   \norm{\partial_x N(w) }_{L^2} d\tau 
	\\
	&\cleq  \int_{t}^{\infty} \tau^{-\frac{3}{2}}  \norm{w}_{L^\infty}^{2} \norm{\partial_x w}_{L^2} d\tau 
	\\
	&\cleq  t^{-\frac{1}{2}} \norm{\varphi}_{L^\infty}^{2} \norm{\partial_x \varphi}_{L^2}(1 + \norm{\varphi}_{L^\infty} \log t ).
\end{aligned}
\end{equation}
Therefore, we see that 
\begin{equation*}
	\begin{aligned}
		\|\Phi(v) - u_{ap}\|_{\mathscr{X}} &\cleq \rho\|\varphi\|_{L^\infty(\cG)}^2
+ \rho^{3} T^{\frac{1}{2}-2\alpha} +   
	T^{-\frac{1}{2}+\alpha}   \norm{\partial_x \varphi}_{L^2}(1 + \norm{\varphi}_{L^\infty} \log T )\\
	&\quad + T^{-\frac{1}{2}+\alpha} \norm{\varphi}_{L^\infty}^{2} \norm{\partial_x \varphi}_{L^2}(1 + \norm{\varphi}_{L^\infty} \log T ).
\end{aligned}
\end{equation*}
Choosing  $1/4 < \alpha < 1/2$, $T$ large enough, and 
$\|\varphi\|_{L^\infty}$ sufficiently small, we see that  
the map $\Phi$ is the map onto $\mathscr{X}_\rho$.
In the same manner, we are able to show that $\Phi$ is the contraction map on $\mathscr{X}_\rho$.
The Banach fixed point theorem implies that $\Phi$ has a unique fixed point in $\mathscr{X}_\rho$, 
which is the solution to the final state problem. This completes the proof of Theorem \ref{thm:FS}. 

\subsection{Initial value problem}
The strategy relies on the argument of Hayashi and Naumkin \cite{HaNa98}. 

Arguing as in the above, we have the following:
\begin{lemma} \label{iv:0d}
Let $u$, $v \in \Sigma_{c}(\mathcal{G})$. 
Then it holds that
\begin{align*}
	&{}\norm{XU(-t)\l( |u|^2u - |v|^2v \r)}_{L^2} \\
	&{}\cleq \l( \norm{u}_{L^\infty} + \norm{v}_{L^\infty} \r)\l( \norm{XU(-t)u}_{L^2} + \norm{XU(-t)v}_{L^2} \r) \norm{u-v}_{L^{\infty}} \\
	&\quad + \l( \norm{u}_{L^\infty}^2 + \norm{v}_{L^\infty}^2 \r) \norm{XU(-t) (u-v)}_{L^2}, \\
	&{}\norm{\partial_x (|u|^2u - |v|^2v)}_{L^2} \\
	{}&\cleq \l( \norm{u}_{L^\infty} + \norm{v}_{L^\infty} \r)\l( \norm{\partial_x u}_{L^2} + \norm{\partial_x v}_{L^2} \r) \norm{u-v}_{L^{\infty}} \\
	&\quad + \l( \norm{u}_{L^\infty}^2 + \norm{v}_{L^\infty}^2 \r) \norm{\partial_x (u-v)}_{L^2}.
\end{align*}
\end{lemma}

Let us recall the local existence of solutions to \eqref{eq:NLS}. Fix $t_0 \in \R$. Set the function space
\begin{align*}
	\mathscr{Z}_{T}^{\varepsilon,A} ={}& \{ \varphi \in C(I_T, (\Sigma_{c} \cap L^{\infty})(\cG))\; ;\; \norm{\varphi}_{\mathscr{Z}_T^{\eps, A}} <\infty \} 
\end{align*}
equipped with
\begin{align*}
	\norm{\varphi}_{\mathscr{Z}}
	=\norm{\varphi}_{\mathscr{Z}_T^{\eps, A}} :={}& \sup_{t\in I_T} (1+|t|)^{-A\varepsilon} \l( \norm{\varphi}_{H^{1}}
	+ \norm{U(-t)\varphi}_{H^{0,1}} \r) \\
	&{}+ \sup_{t\in I_T} (1+|t|)^{\frac12} \norm{\varphi}_{L^\infty} 
\end{align*}
for any $\eps >0$ and all $A>0$, where $I_T = [t_0, t_0 + T]$.

\begin{proposition}[Local existence of solutions] \label{iv:1}
Let $\eps >0$, $A>0$, $t_0 \in \R$, and $K>0$. Assume that $u_0 \in H_c^{1}(\cG)$ satisfies $U(-t_0)u_0 \in H^{0,1}(\cG)$ with $\norm{u_0}_{H^{1}} + \norm{U(-t_0) u_0}_{H^{0,1}} \leq K$. Then there exists $T = T(K) >0$ not depending on $\eps$ and $A$ such that \eqref{eq:NLS} has a unique solution $u \in \mathscr{Z}_{T}^{\varepsilon,A}$ with $u(t_0) = u_0$. Moreover the solution satisfies 
\begin{align}
	\norm{u}_{\mathscr{Z}_{T}^{\varepsilon,A}} \le B_0 K \label{iv:2}
\end{align}
for some constants $B_0 >0$ not depending on $t_0$, $T$, $A$ and $\eps$.
\end{proposition}
\begin{proof}
Set a complete metric space
\[
	\mathscr{Z}_{M}= \mathscr{Z}_{T,M}^{\varepsilon,A}:= \{ f \in L^{\infty}(I_T, (\Sigma_{c} \cap L^{\infty})(\cG))\; ;\; \norm{f}_{\mathscr{Z}} \leq M \}
\]
equipped with the distance function $d(u, v) = \norm{u-v}_{\mathscr{Z}}$,
where the constant $M>0$ will be chosen later.
Let us define the map
\[
	\Psi(u) = U(t-t_0)u_0 - i \lambda \int_{t_0}^t U(t-s) (|u|^2 u)(s)\, ds.
\]
We shall prove that $\Psi(u) \in \mathscr{Z}_{M}$ whenever $u \in \mathscr{Z}_{M}$.
Using Lemma \ref{iv:0}, one easily has 
\begin{align}
\label{eq5.12.0}
	&(1+|t|)^{-A\eps} \l( \norm{\Psi(u(t))}_{H^{1}} + \norm{U(-t)\Psi(u(t))}_{H^{0,1}} \r) \\ \notag
	&\leq (1+|t|)^{-A\eps}\l( \norm{u_0}_{H^{1}} + \norm{U(-t_0)u_0}_{H^{0,1}} \r) \\ \notag
	&\quad + C (1+|t|)^{-A\eps} \int_{t_0}^t \norm{u(s)}_{L^{\infty}}^2\l( \norm{u(s)}_{H^{1}} + \norm{U(-s)u(s)}_{H^{0,1}} \r) ds \\ \notag
	& \leq K + C  \int_{t_0}^t M^3 (1+|s|)^{-1+A\eps} ds \\ \notag
	&= K + C M^3 I_{\varepsilon,A}(t),
\end{align}
for any $t \in I_T$, where $ I_{\varepsilon,A}(t):=\int_{t_0}^t (1+|s|)^{-1+A\eps} ds$. 
Further, similarly to the above, it follows from the dispersive estimate i.e. Lemma \ref{lem4.8},  Lemma \ref{iv:0}, and,  $\| f \|_{L^1} \cleq \| f \|_{H^{0,1}}$ that
\begin{align*}
	 |t|^{\frac12}\norm{\Psi(u(t))}_{L^{\infty}} 
	&\leq |t|^{\frac12}\norm{U(t-t_0) u_0}_{L^{\infty}} + C |t|^{\frac12}\norm{U(t) \int_{t_0}^t U(-s) (|u|^2 u)(s)\, ds}_{L^{\infty}} \\
	&\leq C\norm{U(-t_0) u_0}_{L^{1}} + C\norm{\int_{t_0}^t U(-s) (|u|^2 u)(s)\, ds}_{L^{1}} \\
	&\leq C \norm{U(-t_0) u_0}_{H^{0,1}} + C \int_{t_0}^t \norm{u(s)}_{L^{\infty}}^2 \norm{U(-s) u(s)}_{H^{0,1}}\, ds \\
	&\leq C \norm{U(-t_0) u_0}_{H^{0,1}} + C M^3I_{\varepsilon,A}(t). 
\end{align*}
By Proposition \ref{sob:1}, we also estimate
\begin{align*}
	\norm{\Psi(u(t))}_{L^{\infty}} \leq{}& \norm{U(t-t_0) u_0}_{L^{\infty}} + C \norm{U(t) \int_{t_0}^t U(-s) (|u|^2 u)(s)\, ds}_{L^{\infty}} \\
	\leq{}& C\norm{u_0}_{H^{1}} + C \int_{t_0}^t \norm{u(s)}_{L^{\infty}}^2 \norm{u(s)}_{H^{1}}\, ds \\
	\leq{}& C\norm{u_0}_{H^{1}} + C M^3 I_{\varepsilon,A}(t). 
\end{align*}
These yield
\begin{align}
\label{eq5.13.0}
	&{} \sup_{t \in I_T}(1+|t|)^{\frac12} \norm{\Psi(u(t))}_{L^{\infty}} \leq CK + C M^3 I_{\varepsilon,A}(t_0+T)
\end{align}
Hence, by \eqref{eq5.12.0} and \eqref{eq5.13.0}, we obtain
\begin{align*}
	\norm{\Psi(u)}_{\mathscr{Z}_{T}^{\varepsilon,A}} \leq C_1 K + CM^3 I_{\varepsilon,A}(t_0+T) 
	\leq C_1 \eps_1 + C K^3 I_{\varepsilon,A}(t_0+T) ,
\end{align*}
when taking $M = 2C_1 K$. 
Here, $I_{\varepsilon,A}(t_0+T) \to 0$ as $T \to 0$ uniformly in $A$ and $\varepsilon$. 
Hence $\Psi(u) \in \mathscr{Z}_{M}$ holds as long as $T = T(K) >0$ satisfies
\[
	C K^3 I_{\varepsilon,A}(t_0+T)  \leq C_1K.
\]
By using Lemma \ref{iv:0d} and similar argument to the above, it is possible to show that $\Psi$ is a contraction mapping on $\mathscr{Z}_{M}$. Therefore, we obtain the solution to $u=\Psi(u)$ by the contraction mapping principle. We can also obtain the continuity in time and the uniqueness of the solution.  We omit the details and complete the proof. 
\end{proof}

\begin{corollary}
\label{cor4.4}
Let $A>0$ and $t_0 \in \R$. Then, there exists $\varepsilon_1>0$ such that the following assertion holds: For $0<\varepsilon \leq \varepsilon_1$ and $u_0 \in H_c^{1}(\cG)$ satisfying $U(-t_0)u_0 \in H^{0,1}(\cG)$ with $\norm{u_0}_{H^{1}} + \norm{U(-t_0) u_0}_{H^{0,1}} \leq \varepsilon$, there exists $T = T(\varepsilon) >0$ not depending on $A$ such that \eqref{eq:NLS} has a unique solution $u \in \mathscr{Z}_{T}^{\varepsilon,A}$ with $u(t_0) = u_0$. Moreover the solution satisfies 
\begin{align}
	\norm{u}_{\mathscr{Z}_{T}^{\varepsilon,A}} < \varepsilon^{\frac{1}{2}}.
\end{align}
\end{corollary}

\begin{proof}
Taking $\varepsilon_1>0$ such that $B_0 < \varepsilon_1^{-1/2}$, we have $B_0\varepsilon < \varepsilon^{1/2}$ for any $\varepsilon \in (0,\varepsilon_1]$, where $B_0$ is the constant as in Proposition \ref{iv:1}. Applying Proposition \ref{iv:1} as $K=\varepsilon$, we obtain the statement. 
\end{proof}

In what follows, we fix $A =c_0$, where $c_0$ is given in Lemma \ref{iv:0}. 
We show the following estimate.

\begin{proposition}[Bootstrap estimate]
\label{prop4.5}
There exists $\varepsilon_1 >0$ such that the following holds: Let $T>0$.  If $0<\varepsilon\leq \varepsilon_1$, $u_0 \in H_c^{1}(\cG)$ satisfies $U(-t_0)u_0 \in H^{0,1}(\cG)$ with $\norm{u_0}_{H^{1}} + \norm{U(-t_0) u_0}_{H^{0,1}}\leq \varepsilon$, and the solution $u \in \mathscr{Z}_{T}^{\varepsilon,A}$ of \eqref{eq:NLS} on $[0,T)$ with $u(t_0) = u_0$ satisfies 
\begin{align}
\label{eq4.15}
	\norm{u}_{\mathscr{Z}_{T}^{\varepsilon,A}} \leq \varepsilon^{\frac{1}{2}},
\end{align}
then there exists a constant $C_1$ independent of $T$ and $\varepsilon$ such that
\begin{align*}
	\norm{u}_{\mathscr{Z}_{T}^{\varepsilon,A}} \le C_1 \varepsilon.
\end{align*}
\end{proposition}

We give the proof of this proposition later. Once we obtain this estimate, Theorem \ref{iv:thm1} will be proven as follows.

\begin{proof}[Proof of Theorem \ref{iv:thm1}]
Let $\varepsilon_1$ be smaller than those in Corollary \ref{cor4.4} and Proposition \ref{prop4.5} and satisfy $C_1< \varepsilon_1^{-1/2}$, where $C_1$ is the constant in Proposition \ref{prop4.5}. 
By Corollary \ref{cor4.4}, we have a time $T$ and the unique solution satisfying 
\begin{align}
	\norm{u}_{\mathscr{Z}_{T}^{\varepsilon,A}} < \varepsilon^{\frac{1}{2}}.
\end{align}
We here denote the maximal existence time of its solution by $T_{{\rm max}}(u_0) =: T_{{\rm max}}$.
First, we will show 
\begin{align}
\label{eq4.16}
	\norm{u}_{\mathscr{Z}_{T_{\max}}^{\varepsilon,A}} \leq  \varepsilon^{\frac{1}{2}}
\end{align}
holds. If not, there exists $T_0>0$ such that $\norm{u}_{\mathscr{Z}_{T_0}^{\varepsilon,A}} = \varepsilon^{1/2}$.
It follows from Proposition \ref{prop4.5} and $\varepsilon<\varepsilon_1$ that
\begin{align*}
	\norm{u}_{\mathscr{Z}_{T_0}^{\varepsilon,A}} \le C_1 \varepsilon < \varepsilon^{\frac{1}{2}}.
\end{align*}
This is a contradiction. Thus, $\norm{u}_{\mathscr{Z}_{T_{\max}}^{\varepsilon,A}} \leq  \varepsilon^{1/2}$ holds. 

Next, let us prove $T_{\max}=\infty$ by a contradiction. Suppose that $T_{\max}<\infty$. 
%
Take arbitrarily small $\delta >0$. We set $T_{\delta} = T_{\rm{max}} - \delta < T_{\rm{max}}$.
Then, by \eqref{eq4.16}, 
we have
\[
	\norm{u}_{\mathscr{Z}_{T_{\delta}}^{\eps,A}} \leq \varepsilon^{\frac{1}{2}}.
\]
and thus
\begin{align*}
	\norm{u_1(T_{\delta})}_{H_c^{1,0}} + \norm{U(-T_{\delta})u_1(T_{\delta})}_{H^{0,1}} \leq \varepsilon^{\frac{1}{2}} (1+T_{\delta})^{A\eps} \leq \varepsilon_1^{\frac{1}{2}} (1+T_{{\rm max}})^{A\eps_1} =:K. 
\end{align*}
We note that $K$ is independent of $\delta$. By Proposition \ref{iv:1} as $t_0=T_{\delta}$, we get the solution $\tilde{u}(t)$ on $[T_{\delta}, T_{\delta}+T(K)]$. 
Taking small $\delta>0$, $T_{\delta} + T(K) = T_{\rm{max}} - \delta +T(K) > T_{\rm{max}}$, since we may take $\delta < T(K)$. By the uniqueness and continuity in time,
$u(t)$ exists on $[T_{\rm{max}}, T_{\rm{max}} - \delta +T(K)]$. This is a contradiction.
The proof is completed.
\end{proof}

To prove Proposition \ref{prop4.5}, we need the following lemmas: 

\begin{lemma} \label{iv:3}
Let $u \in C(\R, H_c^{1}(\mathcal{G}))$ and $\alpha \in [0,1/4)$. Then, it holds that
\begin{align*}
	\norm{u(t)}_{L^\infty} \cleq |t|^{-\frac{1}{2}} \norm{\mathcal{F} U(-t) u(t)}_{L^\infty} 
	+ |t|^{-\frac{1}{2}- \alpha} \norm{U(-t)u(t)}_{H^{0,1}}
\end{align*}
for any $t \geq 0$.
\end{lemma}

\begin{proof}
By the factorization property of $U(t)$, we have
\begin{align*}
	u(t)=U(t)U(-t)u(t)
	=\mathcal{MDF}U(-t)u(t) + \mathcal{MDF}(\mathcal{M}-1)U(-t)u(t),
\end{align*}
which implies 
\begin{align*}
	\norm{u(t)}_{L^\infty}
	&\leq \norm{\mathcal{MDF}U(-t)u(t)}_{L^\infty} + \norm{\mathcal{MDF}(\mathcal{M}-1)U(-t)u(t)}_{L^\infty}
	\\
	&\cleq |t|^{-1/2}\norm{\mathcal{F}U(-t)u(t)}_{L^\infty} + |t|^{-1/2}\norm{\mathcal{F}(\mathcal{M}-1)U(-t)u(t)}_{L^\infty}.
\end{align*}
Also, one sees from Lemma \ref{lem4.5} and the H\"{o}lder inequality that 
\begin{align*}
	\norm{\mathcal{F}(\mathcal{M}-1)U(-t)u(t)}_{L^\infty}
	& \cleq \norm{(\mathcal{M}-1)U(-t)u(t)}_{L^1}
	\\
	& \cleq |t|^{-\alpha} \norm{X^{2\alpha}U(-t)u(t)}_{L^1}
	\\
	& \cleq |t|^{-\alpha} \norm{U(-t)u(t)}_{H^{0,1}},
\end{align*}
where we use $\alpha<1/4$ in the last inequality. This completes the proof. 
\end{proof}

We define the operator $\mathcal{J}$ and $\mathcal{L}$ by 
\begin{align*}
	\mathcal{J} &= \mathcal{J}(t) :=U(t)XU(-t), 
	\\
	\mathcal{L}&= i\partial_t + \Delta_{K}.
\end{align*} 

\begin{lemma} \label{iv:4}
We have $[\mathcal{J}, \mathcal{L}] := \mathcal{J}\mathcal{L}-\mathcal{L}\mathcal{J}= 0$.
\end{lemma}

\begin{proof}
Noting that $U(t)=e^{it\Delta_{K}}$ and $i\partial_{t} U(t)=e^{it\Delta_{K}} (-\Delta_{K})$, we calculate 
\begin{align*}
	\mathcal{L}\mathcal{J}f &=(i\partial_t + \Delta_{K})(U(t)XU(-t)f)
	\\
	&=i\partial_t (U(t)XU(-t)f)+ U(t)\Delta_{K} XU(-t)f
	\\
	&= -(U(t)\Delta_{K}  XU(-t)f)+ (U(t)X(i\partial_t )U(-t)f)+ U(t)\Delta_{K} XU(-t)f
	\\
	&=(U(t)X(i\partial_t )U(-t)f)
	\\
	&=(U(t)X(\Delta_{K})U(-t)f) + (U(t)XU(-t)(i\partial_t )f)
	\\
	&=U(t)XU(-t) (i\partial_t + \Delta_{K})f=\mathcal{J}\mathcal{L}f.
\end{align*}
The proof is completed. 
\end{proof}

%

We divide the proof of Proposition \ref{prop4.5} into three parts as follows:

\begin{proposition} \label{iv:pro1}
There exists $\eps_1>0$ such that the following assertion holds: 
If $\varepsilon$, $u_0$, and $u$ satisfy the assumption in Proposition \ref{prop4.5}, 
then the estimate
\begin{align*}
	&{}(1+|t|)^{-A \eps} \l( \norm{u(t)}_{H^{1}} + \norm{U(-t)u(t)}_{H^{0,1}} \r) \leq \varepsilon 
\end{align*}
is valid for any $t \in [0,T]$.
\end{proposition}

\begin{proof}
By Lemma \ref{iv:4}, multiplying the equation \eqref{eq:NLS} by the operator $\mathcal{J}$, 
%
\begin{align*}
	(i\partial_t + \Delta_K)(\mathcal{J}u) = -\lambda \mathcal{J}( |u|^2u).
\end{align*}
The equation is rewritten as an integral equation
\begin{align*}
	\mathcal{J}(t)u(t)= U(t)\mathcal{J}(0)u_{0} + i\lambda  \int_{0}^{t} U(t-s) \mathcal{J}(s)(|u|^2u)(s) ds.
\end{align*}
We see from \eqref{eq4.15} and Lemma \ref{iv:0} that
\begin{align*}
	\norm{XU(-t)u(t)}_{L^2}
	={}& \norm{\mathcal{J}(t)u(t)}_{L^2} \\
	\leq{}&  \norm{U(t)\mathcal{J}(0)u_{0}}_{L^2} +\int_{0}^{t} \norm{\mathcal{J}(s)(|u|^2u)(s)}_{L^2} ds
	\\
	\leq{}& \norm{X u_{0}}_{L^2} + c_0 \int_{0}^{t} \norm{u(s)}_{L^\infty}^2 \norm{XU(-s)u(s)}_{L^2} ds \\
	\leq{}& \norm{X u_{0}}_{L^2} + c_0 \varepsilon  \int_{0}^{t} (1+|s|)^{-1} \norm{XU(-s)u(s)}_{L^2} ds.
\end{align*}
We here note $A = c_0$. 
Hence, the Gronwall inequality gives us 
\begin{align*}
	\norm{XU(-t)u(t)}_{L^2} \leq \norm{X u_{0}}_{L^2} (1+|t|)^{A\eps}, 
\end{align*}
which yields
\begin{align}
	(1+|t|)^{-A\eps} \norm{XU(-t)u(t)}_{L^2} \leq \norm{X u_{0}}_{L^2}  \label{iv:5}
\end{align}
for any $t \in [0,T]$.
Arguing as in the above, since 
\begin{align*}
	\norm{u(t)}_{H^{1}} 
	\leq{}&  \norm{U(t)u_{0}}_{H^{1}} +\int_{0}^{t} \norm{(|u|^2u)(s)}_{H^{1}} ds
	\\
	\leq{}& \norm{u_{0}}_{H^{1}} +  c_0  \varepsilon  \int_{0}^{t} (1+|s|)^{-1} \norm{u(s)}_{H^{1}} ds,
\end{align*}
we deduce that
\begin{align}
	(1+|t|)^{-A\eps} \norm{u(t)}_{H^{1}} \leq \norm{u_{0}}_{H^{1}}  \label{iv:6}.
\end{align}
Collecting \eqref{iv:5} and \eqref{iv:6}, one obtains the desired estimate. 
\end{proof}

\begin{proposition} \label{iv:pro2} 
There exists $\eps_1 \in (0,1)$ such that the following assertion holds: 
If $\varepsilon$, $u_0$, and $u$ satisfy the assumption in Proposition \ref{prop4.5}, 
then the estimate
\begin{align*}
	&{}(1+|t|)^{\frac12} \norm{u(t)}_{L^{\infty}} \leq C\varepsilon 
\end{align*}
is valid for any $t \in [0,T]$.
\end{proposition}

\begin{proof}
Take $\eps_1 < \min\{1, \alpha (3A)^{-1}\}$ for some $\alpha \in (0, 1/4)$. 
When $t \leq 1$, by \eqref{iv:6} and Proposition \ref{sob:1}, we easily show
\[
	(1+|t|)^{\frac12} \norm{u(t)}_{L^{\infty}} \leq C(1+|t|)^{\frac12+A\eps} \norm{u_0}_{H^{1}} \leq C\norm{u_0}_{H^{1}} \leq C\varepsilon.
\]
We shall consider the case $t \ge 1$. 
Combining Lemma \ref{iv:3} with Proposition \ref{iv:pro1}, one obtains
\begin{align}
	\begin{aligned}
	\norm{u(t)}_{L^\infty} \leq{}& C|t|^{-\frac{1}{2}} \norm{\mathcal{F} U(-t) u(t)}_{L^\infty} 
	+ C|t|^{-\frac{1}{2}- \alpha} \norm{U(-t)u(t)}_{H^{0,1}} \\
	\leq{}& C|t|^{-\frac{1}{2}} \norm{\mathcal{F} U(-t) u(t)}_{L^\infty} 
	+ C\varepsilon|t|^{-\frac{1}{2}- \alpha+A\eps} 
	\end{aligned}
	\label{iv:8}
\end{align}
Let us handle the first term in the above last line.
It follows from \eqref{eq:NLS} that
\begin{align}
	i\partial_t (U(-t)u) = -\lambda U(-t)(|u|^2u). \label{iv:7}
\end{align}
A computation shows 
\begin{align*}
	U(-t)(|u|^2 u)
	&{}=\mathcal{M}(-t) \mathcal{F}^{-1} \mathcal{D}^{-1}\mathcal{M}(-t)(|u|^2 u)
	\\
	&{}=\mathcal{M}(-t) \mathcal{F}^{-1} \mathcal{D}^{-1}(|\mathcal{M}(-t)u|^2 \mathcal{M}(-t)u)
	\\
	&{}=(2t)^{-1}\mathcal{M}(-t) \mathcal{F}^{-1} (|\mathcal{D}^{-1}\mathcal{M}(-t)u|^2 \mathcal{D}^{-1}\mathcal{M}(-t)u).
\end{align*}
Let $v$ satisfy $U(t)v=u$, namely, $v=U(-t)u$. Then, $\mathcal{D}^{-1}\mathcal{M}(-t)u =\mathcal{F}\mathcal{M}(t)v$. 
Hence, we have 
\begin{align*}
	U(-t)(|u|^2 u) ={}& (2t)^{-1} \mathcal{M}(-t) \mathcal{F}^{-1} (|\mathcal{F}\mathcal{M}(t)v|^2\mathcal{F}\mathcal{M}(t)v) \\
	={}& (2t)^{-1}(\mathcal{M}(-t)-1) \mathcal{F}^{-1} (|\mathcal{F}\mathcal{M}(t)v|^2\mathcal{F}\mathcal{M}(t)v) \\
	&{}+ (2t)^{-1}\mathcal{F}^{-1}\{ |\mathcal{F}\mathcal{M}(t)v|^2\mathcal{F}\mathcal{M}(t)v - |\mathcal{F}v|^2\mathcal{F}v \} \\
	&{}+ (2t)^{-1}\mathcal{F}^{-1} (|\mathcal{F}v|^2\mathcal{F}v).
\end{align*}
Therefore, plugging the above into \eqref{iv:7} and taking $\mathcal{F}$, it is deduced that 
\begin{align*}
	&{}i (\cF v)_t + \frac{\lambda}{2} t^{-1} (|\mathcal{F}v|^2\mathcal{F}v)
	\\
	&= -\frac{\lambda}{2} t^{-1} \mathcal{F} (\mathcal{M}(-t)-1) \mathcal{F}^{-1} (|\mathcal{F}\mathcal{M}(t)v|^2\mathcal{F}\mathcal{M}(t)v)
	\\
	&\quad -\frac{\lambda}{2} t^{-1} \{ |\mathcal{F}\mathcal{M}(t)v|^2\mathcal{F}\mathcal{M}(t)v - |\mathcal{F}v|^2\mathcal{F}v \}.
\end{align*}
We here define
\begin{align*}
	I_1(t) &{}= \mathcal{F}(\mathcal{M}(-t)-1) \mathcal{F}^{-1} (|\mathcal{F}\mathcal{M}(t)v|^2\mathcal{F}\mathcal{M}(t)v),
	\\
	I_2(t) &{}= |\mathcal{F}\mathcal{M}(t)v|^2\mathcal{F}\mathcal{M}(t)v - |\mathcal{F}v|^2\mathcal{F}v.
\end{align*}
Let us also introduce a phase translation as follows:
\begin{align*}
	 w= \mathcal{F}^{-1}( \mathcal{B}(t) \mathcal{F}v),
\end{align*}
where 
\begin{align*}
	\mathcal{B}(t)&{}=
	\begin{pmatrix}
	B_1 & 0 & \cdots &0
	\\
	0 & B_2 & \cdots &0
	\\
	\vdots & & & \vdots
	\\
	0 & 0 & \cdots & B_n
	\end{pmatrix}
	, \quad
	B_j(t) = \exp \left( - \frac{\lambda}{2} i\int_{1}^{t} |(\mathcal{F}v)_j|^2 \frac{d\tau}{\tau} \right).
\end{align*}
Thus, we obtain
\begin{align*}
	i \partial_t (\mathcal{F}w) = -\frac{\lambda}{2} t^{-1} \mathcal{B}(t) (I_1+I_2)(t).
\end{align*}
Integrating the above on $[1, t]$, noting that $\mathcal{F}w(1) = \mathcal{F}v(1)=\mathcal{F}U(-1)u(1)$, one has
\begin{align}
	\mathcal{F}w(t) =\mathcal{F}U(-1)u(1)+ \frac{\lambda}{2} i \int_{1}^{t} \mathcal{B}(\tau)  (I_1+I_2)(\tau) \frac{d \tau}{\tau}. \label{iv:12}
\end{align}
We shall estimate $I_1$. By means of Lemma \ref{lem4.5} and Sobolev and Hausdorff--Young inequalities, one sees that
\begin{align}
	\begin{aligned}
	\norm{I_1(t)}_{L^{\infty}} ={}& \norm{\mathcal{F}(\mathcal{M}(-t)-1) \mathcal{F}^{-1} (|\mathcal{F}\mathcal{M}(t)v|^2\mathcal{F}\mathcal{M}(t)v)}_{L^{\infty}} \\
	\cleq{}& \norm{(\mathcal{M}(-t)-1) \mathcal{F}^{-1} (|\mathcal{F}\mathcal{M}(t)v|^2\mathcal{F}\mathcal{M}(t)v)}_{L^{1}} \\
	\cleq{}& t^{-\alpha} \norm{(1+X) \mathcal{F}^{-1} (|\mathcal{F}\mathcal{M}(t)v|^2\mathcal{F}\mathcal{M}(t)v)}_{L^{2}} \\
	\cleq{}& t^{-\alpha} \norm{\mathcal{F}\mathcal{M}(t)v}_{L^{\infty}}^2 \norm{\mathcal{F}\mathcal{M}(t)v}_{L^{2}} \\
	&{}+ t^{-\alpha} \norm{\mathcal{F}_{c}^{-1} \partial_x (|\mathcal{F}\mathcal{M}(t)v|^2\mathcal{F}\mathcal{M}(t)v)}_{L^{2}} \\
	\cleq{}& t^{-\alpha} \norm{\mathcal{F}\mathcal{M}(t)v}_{H^1}^2 \norm{\mathcal{F}\mathcal{M}(t)v}_{L^{2}} \\
	&{}+  t^{-\alpha} \norm{\partial_x \mathcal{F}\mathcal{M}(t)v}_{L^{2}} \norm{\mathcal{F}\mathcal{M}(t)v}_{L^{2}}^2 \\
	\cleq{}& t^{-\alpha} \norm{v}_{H^{0,1}}^3.
	\end{aligned}
	\label{iv:10}
\end{align}
Further, the estimation of $I_2$ is as follows:
\begin{align}
	\begin{aligned}
	\norm{I_2(t)}_{L^{\infty}} ={}& \norm{|\mathcal{F}\mathcal{M}(t)v|^2\mathcal{F}\mathcal{M}(t)v - |\mathcal{F}v|^2\mathcal{F}v}_{L^{\infty}} \\
	\cleq{}& \l( \norm{\mathcal{F}\mathcal{M}(t)v}_{L^{\infty}}^2 + \norm{\mathcal{F}v}_{L^{\infty}}^2 \r) \norm{\mathcal{F}\mathcal{M}(t)v - \mathcal{F}v}_{L^{\infty}} \\
	\cleq{}& \norm{v}_{L^{1}}^2 \norm{(\mathcal{M}(t)-1)v}_{L^1} \\
	\cleq{}& t^{-\alpha} \norm{v}_{H^{0,1}}^3.
	\end{aligned}
	\label{iv:11}
\end{align}
Moreover, we have
\begin{align*}
	 \norm{\mathcal{F}U(-1)u(1)}_{L^\infty} \leq  \norm{U(-1)u(1)}_{L^1}  \leq  \norm{U(-1)u(1)}_{H^{0,1}} \leq C \varepsilon 
\end{align*}
by Proposition \ref{iv:pro1}. 
Combining these above with Proposition \ref{iv:pro1}, we reach to
\begin{align}
	\begin{aligned}
	\norm{\cF U(-t)u(t)}_{L^{\infty}} ={}& \norm{\cF w(t)}_{L^{\infty}} \\
	\leq{}& \norm{\mathcal{F}U(-1)u(1)}_{L^\infty}+ C \int_1^{t} \tau^{-\alpha-1} \norm{U(-\tau)u(\tau)}_{H^{0,1}}^3\, d\tau \\
	\leq{}&C\varepsilon+ C\varepsilon^3 \int_1^{t} \tau^{-\alpha-1+3A\eps}\, d\tau \\
	\leq{}& C_{\eps_1, \alpha} \varepsilon
	\end{aligned}
	\label{iv:9}
\end{align}
for any $|t| \geq 1$. 
Therefore collecting \eqref{iv:8} and \eqref{iv:9}, we conclude
\begin{align*}
	(1+|t|)^{\frac12} \norm{u(t)}_{L^{\infty}} \leq C ( 1+ t^{-\alpha + A\eps})\varepsilon \leq C\varepsilon
\end{align*}
for any $t \geq 1$. This completes the proof. 
\end{proof}

%

\begin{proof}[Proof of Proposition \ref{prop4.5}]
The consequence immediately follows from Proposition \ref{iv:pro2} and Proposition \ref{iv:pro1}.
\end{proof}

Let us move on to the proof of Theorem \ref{iv:thm3}.
\begin{proof}[Proof of Theorem \ref{iv:thm3}]
We shall first show \eqref{iv:thm4}.
Combining \eqref{iv:12} with \eqref{iv:10} and \eqref{iv:11}, we see from Proposition \ref{iv:pro1} that 
\begin{align}
	\begin{aligned}
	\norm{\cF w(t) - \cF w(s)}_{L^{\infty}} \leq{}& C \int_{s}^{t} \l( \norm{I_1(\tau)}_{L^{\infty}} + \norm{I_2(\tau)}_{L^{\infty}} \r) \frac{d \tau}{\tau} \\
	\leq{}& C \int_s^t \tau^{-\alpha-1} \norm{v(\tau)}_{H^{0,1}}^3\, d\tau \\
	\leq{}& C \eps^3 \int_s^t \tau^{-\alpha-1+3A\eps}\, d\tau \\
	\leq{}& C \eps \l( t^{-\alpha+3A\eps} + s^{-\alpha+3A\eps} \r) \to 0
	\end{aligned}
	\label{iv:13}
\end{align}
as $t$, $s \to \infty$. Using Lemmas \ref{lem4.5} and \ref{iv:0}, 
one also estimates
\begin{align*}
	\norm{I_1(t)}_{L^{2}} \cleq{}& t^{-\frac12} \norm{X \mathcal{F}^{-1} (|\mathcal{F}\mathcal{M}(t)v|^2\mathcal{F}\mathcal{M}(t)v)}_{L^{2}} \\
	\cleq{}& t^{-\frac12} \norm{\mathcal{F}_{c}^{-1} \partial_x (|\mathcal{F}\mathcal{M}(t)v|^2\mathcal{F}\mathcal{M}(t)v)}_{L^{2}} \\
	\cleq{}& t^{-\frac12} \norm{v}_{H^{0,1}}^3, \\
	\norm{I_2(t)}_{L^{2}} \cleq{}& \l( \norm{\mathcal{F}\mathcal{M}(t)v}_{L^{\infty}}^2 + \norm{\mathcal{F}v}_{L^{\infty}}^2 \r) \norm{\mathcal{F}(\mathcal{M}(t)-1)v}_{L^{2}} \\
	\cleq{}& t^{-\frac12} \norm{v}_{H^{0,1}}^3.
\end{align*}
Hence it holds that
\begin{align}
	\begin{aligned}
	\norm{\cF w(t) - \cF w(s)}_{L^2} \leq{}& C \int_{s}^{t} \l( \norm{I_1(\tau)}_{L^{2}} + \norm{I_2(\tau)}_{L^{2}} \r) \frac{d \tau}{\tau} \\
	\leq{}& C \int_s^t \tau^{-\frac{3}{2}} \norm{v(\tau)}_{H^{0,1}}^3\, d\tau \\
	\leq{}& C \eps \l( t^{-\frac12+3A\eps} + s^{-\frac12+3A\eps} \r) \to 0
	\end{aligned}
	\label{iv:13a}
\end{align}
as $t$, $s \to \infty$. 
Therefore, there exists $W \in L^{2}(\cG) \cap L^{\infty}(\cG)$ such that 
\[
	\norm{\cF w(t) - W}_{L^{2} \cap L^{\infty}} \to 0
\]
as $t \to \infty$. Taking $s \to \infty$, \eqref{iv:13} and \eqref{iv:13a} reach to
\begin{align}
	\norm{\cF w(t) - W}_{L^{2} \cap L^{\infty}} \leq C \eps t^{-\alpha +3A \eps}, \label{iv:14}
\end{align}
since $\alpha<1/4$, which implies 
\begin{align}
	\norm{\mathcal{B}(t) \cF U(-t)u(t)  - W}_{L^{2} \cap L^{\infty}} \leq C \eps t^{-\alpha +3A \eps}. \label{iv:15}
\end{align}
Thus we have \eqref{iv:thm4}. Let us prove \eqref{iv:thm5}.
We here define
\[
	\Theta(t) := \frac{\lambda}2 \int_1^t \l( |\cF w(\tau)|^2 - |\cF w(t)|^2 \r) \frac{d \tau}{\tau}.
\]
A direct computation shows
\begin{align}
	\begin{aligned}
	\Theta(s) - \Theta(t) ={}& \frac{\lambda}2 \int_t^s \l( |\cF w(\tau)|^2 - |\cF w(s)|^2 \r) \frac{d \tau}{\tau} \\
	&{}+ \frac{\lambda}2 \l( |\cF w(t)|^2 - |\cF w(s)|^2 \r) \log t
	\end{aligned}
	\label{iv:16}
\end{align}
for any $1<t<\tau<s$. Since we see from \eqref{iv:9} and \eqref{iv:13} that 
\begin{align*}
	\l| |\cF w(\tau)|^2 - |\cF w(s)|^2 \r| \leq{}&( |\cF w(\tau) |+| \cF w(s)|)|\cF w(\tau) - \cF w(s)| \\
	\leq{}& C \eps \l( \tau^{-\alpha+3A\eps} + s^{-\alpha+3A\eps} \r) \\
	\leq{}& C \eps \tau^{-\alpha+3A\eps}
\end{align*}
for any $\tau < s$, \eqref{iv:16} implies that
\begin{align}
	\begin{aligned}
	|\Theta(s) - \Theta(t)| \leq{}& C \eps \int_t^s \tau^{-\alpha+3A\eps-1}\, d\tau + C \eps t^{-\alpha+3A\eps} \log t \\
	\leq{}& C \eps \l( t^{-\alpha+3A\eps} - s^{-\alpha+3A\eps}  \r) + C \eps t^{-\alpha+3A\eps }\log t
	\end{aligned}
	\label{iv:17}
\end{align}
for any $t<s$. Hence there exists a real valued fonction $\Psi \in L^{\infty}(\cG)$ such that $\norm{\Theta(t) - \Psi}_{L^{\infty}} \to 0$ as $t \to \infty$. As for \eqref{iv:17}, taking $s \to \infty$, 
\begin{align}
	\norm{\Psi - \Theta(t)}_{L^{\infty}} \leq  C \eps t^{-\alpha+3A\eps} \log t. \label{iv:18}
\end{align}
By the definition of $\Theta(t)$, we have
\begin{align*}
	&\frac{\lambda}2 \int_1^t |\cF w(\tau)|^2 \frac{d \tau}{\tau} - \frac{\lambda}2 |W|^2 \log t - \Psi 
	\\
	&= (\Theta(t) - \Psi) + \frac{\lambda}{2} \l( |\cF w(t)|^2 - |W|^2 \r) \log t.
\end{align*}
Thus, one sees from \eqref{iv:9}, \eqref{iv:14}, and \eqref{iv:18} that
\begin{align*}
	&{}\norm{\frac{\lambda}2 \int_1^t |\cF w(\tau)|^2 \frac{d \tau}{\tau} - \frac{\lambda}2 |W|^2 \log t - \Psi}_{L^{\infty}} \\ 
	&\leq \norm{\Theta(t) - \Psi}_{L^{\infty}} + C \left(\norm{\cF w(t)}_{L^\infty} +\norm{ W}_{L^{\infty}}\right) \norm{\cF w(\tau) - W}_{L^{\infty}} \log t \\
	&\leq C \eps t^{-\alpha+3A\eps} \log t,
\end{align*}
where we note that $\| W\|_{L^{\infty}} \leq C \varepsilon$ by \eqref{iv:9} and \eqref{iv:14}. 
Thus \eqref{iv:thm5} has been proven. Further, noting $|1- e^{i\theta}| \cleq |\theta|$ for any $\theta \in \mathbb{R}$, combining \eqref{iv:15} with \eqref{iv:thm5}, it holds that 
\begin{align*}
	&{}\norm{(\cF v)_j(t) - \exp \left(  i \frac{\lambda}2 |W_j|^2 \log t + i \Psi_j \right) W_j}_{L^{\infty} \cap L^2(e_j)} \\
	&\leq \norm{(\cF v)_j(t) - \exp \left( \frac{\lambda}{2} i\int_{1}^{t} |(\mathcal{F}v)_j|^2 \frac{d\tau}{\tau} \right) W_j}_{L^{\infty} \cap L^2(e_j)} \\
	&\quad + \norm{\left\{\exp \left(  i \frac{\lambda}2 |W_j|^2 \log t + i \Psi_j - \frac{\lambda}{2} i\int_{1}^{t} |(\mathcal{F}v)_j|^2 \frac{d\tau}{\tau} \right) -1\right\} W_j}_{L^{\infty} \cap L^2(e_j)} \\
	&\leq C\eps t^{-\alpha+3A\eps} + C \norm{\frac{\lambda}2 |W_j|^2 \log t + i \Psi_j - \frac{\lambda}{2} i\int_{1}^{t} |(\mathcal{F}v)_j|^2 \frac{d\tau}{\tau}}_{L^{\infty}(e_j)}\norm{W_j}_{L^{\infty} \cap L^2(e_j)} \\
	&\leq C\eps t^{-\alpha+3A\eps} + C\eps t^{-\alpha+3A\eps}\log t 
	\\
	&\leq C\eps t^{-\alpha+3A\eps} \log t
\end{align*}
for any $t \geq 1$ and each $j = 1, 2, \ldots, n$. Hence we conclude \eqref{iv:thm6}. 
In terms of the asymptotic formula \eqref{iv:thm7}, from \eqref{iv:thm6} and the estimate in Proposition \ref{iv:pro1}, it is established that  
\begin{align*}
	u(t) 
	={}& \cM \cD \cF v(t) + \cM \cD \cF (\cM -1) v(t) \\
	={}& \cM \cD \exp \left(  i \frac{\lambda}2 |W|^2 \log t + i \Psi \right) W \\
	&{}+ \cM \cD \l( \cF v(t) - \exp \left(  i \frac{\lambda}2 |W|^2 \log t + i \Psi \right) W \r) + \cM \cD \cF (\cM -1) v(t) \\
	={}& \frac{1}{(2it)^{\frac12}} W\l( \frac{x}{2t} \r) \exp \left( \frac{i|x|^2}{4t} + i \frac{\lambda}2 \l|W\l( \frac{x}{2t} \r)\r|^2 \log t + i \Psi\l( \frac{x}{2t} \r) \right)  \\
	&{}+ O\l( \eps t^{-\frac12 -\alpha +3A\eps } \log t \r) + O\l( \eps t^{-\frac12 -\alpha+A\varepsilon} \r).
\end{align*}
This completes the proof.
\end{proof}

\subsection{Failure of the scattering for $1 \le p \le 2$}
\label{sec6}

In this section, we show the failure of scattering for \eqref{NLS^p} when $1<p\leq 2$. 
Our proof is based on the standard argument by Cazenave \cite{Caz}.

To prove Theorem \ref{thm2.4}, we derive a contradiction supposing $v_{+}\neq 0$. Set $F(\varphi):=\lambda |\varphi|^p \varphi$. 
Let $\varphi$ be as in \cite[Lemma 8]{AIM20}. Namely, there exists $\delta>0$ such that $\tbra{F(\mathcal{F} v_+)}{\mathcal{F}\varphi}_{L^2(e_j)} < -\delta$. 
Multiplying $\overline{w}:=\overline{e^{it\Delta_{K}}\varphi}$ to \eqref{NLS^p}, integrating it on $e_j$, and taking the summation for $j$, we obtain the following weak formula:
\begin{align}
\label{eq4.29}
	i \frac{d}{dt} \tbra{u}{w}_{L^2(\mathcal{G})}  +\tbra{F(u)}{w}_{L^2(\mathcal{G})}=0.
\end{align}

\begin{remark}
Before taking the summation, we have the term 
\begin{align*}
	u_j(t,0+) \overline{\partial_x w_j (t,0+)} +  \partial_x u_j(t,0+) \overline{ w_j (t,0+)}
\end{align*}
from the boundary. 
Taking the summation, this term disappears, namely, 
\begin{align*}
	\sum_{j=1}^{n} \left( u_j(t,0+) \overline{\partial_x w_j (t,0+)} +  \partial_x u_j(t,0+) \overline{ w_j (t,0+)}\right)=0
\end{align*}
since $u$ and $w$ satisfiy the Kirchhoff boundary condition.  
\end{remark}

Now, setting $\widetilde{u}=\mathcal{D}^{-1}\mathcal{M}^{-1}u$, and $\widetilde{w}=\mathcal{D}^{-1}\mathcal{M}^{-1}w$, then we have 
\begin{align*}
	\tbra{F(u)}{w}_{L^2(\mathcal{G})} = t^{-\frac{p}{2}} \tbra{F(\widetilde{u})}{\widetilde{w}}_{L^2(\mathcal{G})}
\end{align*}
since $F$ is gauge invariant. We note that $\widetilde{w}=\mathcal{F}\mathcal{M} \varphi$ by the factorization formula of $e^{it\Delta_{K}}$. 

Taking real part and integrating \eqref{eq4.29} on time interval $[T,\tau]$, where $T$ is large enough, 
\begin{align}
\label{eq6.1}
	\int_{T}^{\tau} \re (i \partial_t \tbra{u}{w}_{L^2(\mathcal{G})}) dt 
	\cgeq  - \int_{T}^{\tau}  t^{-\frac{p}{2}} \re \tbra{F(\widetilde{u})}{\widetilde{w}}_{L^2(\mathcal{G})} dt.
\end{align}
The left hand side is bounded for $\tau$ since the $L^2$-norms of $u$ and $w$ conserve (see \cite{GrIg}). Therefore, we will show that the right hand side is unbounded if $v_{+}\neq 0$. 
We have the following lemma.

\begin{lemma}
\label{lem6.3}
Let $1\leq p\leq 2$. If $u$ is a forward-global solution of \eqref{NLS^p} satisfying
\begin{align*}
	\left\| e^{-it\Delta_{K}} u\left( t \right) - v_{+} \right\|_{\Sigma\left( \mathcal{G} \right)} \to 0 \quad \left( t \to \infty \right)
\end{align*}
for some  $v_{+} \in \Sigma\left( \mathcal{G} \right)$, then 
\begin{align}
\label{eq3.3}
	F\left( \widetilde{u} \right)= F\left( \mathcal{F} v_{+} \right) + o\left( 1 \right) \text{ in } L^{1} \left( \mathcal{G}\right) \text{ as } t \to \infty.
\end{align}
\end{lemma}

\begin{proof}
By the H\"{o}lder inequality, we get
\begin{align*}
\left\| F\left( \widetilde{u} \right) -F\left( \mathcal{F}v_{+} \right) \right\|_{L^{1}}
	&\leq C\left\| \left| \widetilde{u} - \mathcal{F}v_{+} \right| \left( \left| \widetilde{u} \right|^{p} + \left| \mathcal{F}v_{+} \right|^{p} \right) \right\|_{L^{1}}\\
	&\leq C\left\| \widetilde{u} - \mathcal{F}v_{+} \right\|_{L^{\frac{2}{2-p}}}\left( \left\| \widetilde{u} \right\|_{L^{2}}^{p} + \left\| \mathcal{F}v_{+} \right\|_{L^{2}}^{p} \right)\\
	&\leq C\left\| \widetilde{u} - \mathcal{F}v_{+} \right\|_{H^{1}}\left( \left\| e^{-it\Delta_{K}}u - v_{+} \right\|_{L^{2}}^{p} + \left\| v_{+} \right\|_{L^{2}}^{p} \right).
\end{align*}
Since $e^{-it\Delta_{K}}=\mathcal{M}^{-1}\mathcal{F}^{-1}\mathcal{D}^{-1}\mathcal{M}^{-1}$, we obtain 
\begin{align*}
\left\| \widetilde{u} - \mathcal{F} v_{+} \right\|_{H^{1}}
	&\leq C\left\| \left( 1+\partial_{x} \right)\mathcal{F}\left( \mathcal{F} ^{-1}\mathcal{D}^{-1}\mathcal{M}^{-1}u - v_{+} \right) \right\|_{L^{2}}\\
	&\leq C\left\| \left(\mathcal{F}-i \mathcal{F}_{c}X \right) \left( \mathcal{F}^{-1}\mathcal{D}^{-1}\mathcal{M}^{-1}u - v_{+} \right) \right\|_{L^{2}}\\
	&\leq C\left\| \mathcal{M}^{-1}\mathcal{F}^{-1}\mathcal{D}^{-1}\mathcal{M}^{-1}u - \mathcal{M}^{-1}v_{+} \right\|_{H^{0, 1}}\\
	&\leq C\left\| e^{-it\Delta_{K}}u - v_{+} \right\|_{H^{0, 1}} + \left\| \left( 1-\mathcal{M}^{-1} \right) v_{+} \right\|_{H^{0, 1}}.
\end{align*}
By the assumption, we have $\left\| e^{-it\Delta_{K}}u - v_{+} \right\|_{H^{0, 1}} \to 0$ as $t \to \infty$. 
We also have $\| ( 1-\mathcal{M}^{-1} )v_{+}\|_{H^{0, 1}} \to 0$ as $t \to \infty$ by the Lebesgue dominated convergence theorem. Thus, we get
\begin{align*}
\left\| \widetilde{u} - \mathcal{F}v_{+} \right\|_{H^{1}} \to 0 \quad \text{as} \quad t \to 0.
\end{align*}
Therefore, we find
\begin{align*}
\left\| F\left( \widetilde{u} \right) -F\left( \mathcal{F}v_{+} \right) \right\|_{L^{1}}\to 0 \quad \text{as} \quad t \to 0. 
\end{align*}
This completes the proof.
\end{proof}

By the same method as in  \cite[Lemma 4]{AIM20}, we obtain
\begin{align*}
	\widetilde{w} = \mathcal{F} \varphi + o(1) \text{ in }   L^{\infty} (\mathcal{G}) \text{ as } t \to \infty.
\end{align*}
Combining this with  Lemma \ref{lem6.3}, we obtain 
\begin{align*}
	\tbra{F(\widetilde{u})}{\widetilde{w}}_{L^2(\mathcal{G})} = \tbra{F(\mathcal{F} v_+)}{\mathcal{F}\varphi}_{L^2(\mathcal{G})} + o(1) \text{ as } t \to \infty.
\end{align*}
Indeed, 
\begin{align*}
	&|\tbra{F(\widetilde{u})}{\widetilde{w}}_{L^2(\mathcal{G})} - \tbra{F(\mathcal{F} v_+)}{\mathcal{F}\varphi}_{L^2(\mathcal{G})} |
	\\
	&\leq  \| F(\widetilde{u}) - F(\mathcal{F} v_+ ) \|_{L^1} \| \widetilde{w}\|_{L^\infty}
	+\| F(\mathcal{F} v_+)\|_{L^1} \| \widetilde{w}-\mathcal{F}\varphi\|_{L^\infty}
	\\
	& \to 0 
\end{align*}
as $t \to \infty$ since $\| \widetilde{w}\|_{L^\infty} = \|\mathcal{F}\mathcal{M} \varphi \|_{L^\infty}\leq \|\varphi\|_{L^1}$ and $\| F(\mathcal{F} v_+)\|_{L^1} =\|\mathcal{F} v_+\|_{L^{p+1}} \leq \|v_+\|_{L^{1+1/p}} \leq \|v_+\|_{\Sigma}$. 
Moreover, by the definition of $\varphi$, there exists $\delta>0$ such that
\begin{align*}
	\re \tbra{F(\widetilde{u})}{\widetilde{w}}_{L^2(e_j)} < -\delta
\end{align*}
for large $t>T$. Thus, \eqref{eq6.1} implies that
\begin{align*}
	C \cgeq \delta \int_{T}^{\tau} t^{-\frac{p}{2}} dt  \to \infty
\end{align*}
as $\tau \to \infty$ when $p\leq 2$. 
This derives a contradiction. We complete the proof of Theorem \ref{thm2.4}.  

\appendix

\section{Proof of Preliminaries}
\label{appA}

The Schr\"{o}dinger evolution group with the Kirchhoff boundary condition
is obtained by \cite{AdCaFiNo11} (see also \cite{GrIg}) as follows:
\begin{align}
\label{eqA.1}
	U(t)=e^{it\Delta_K} =
	(\mathcal{U}_{t}^{-} - \mathcal{U}_{t}^{+})  I_n
	+\frac{2}{n}
	\mathcal{U}_t^{+} J_n,
\end{align}
where 
\begin{equation*}
[\mathcal{U}_t^\pm f](x) :=\int_0^\infty \frac{1}{\sqrt{4\pi it}}e^{\frac{i|x\pm y|^2}{4t}} f(y) dy. 
\end{equation*}
This formula and the definition of the Fourier transform $\mathcal{F}$ imply the factorization formula in Proposition \ref{prop4.1} as follows.

\begin{proof}[Proof of Proposition \ref{prop4.1}]
First of all, we have 
\begin{align*}
	[\mathcal{U}_t^\pm f](x) 
	=e^{\frac{i|x|^2}{4t}}  \frac{1}{\sqrt{2it}} \frac{1}{\sqrt{2\pi}} \int_0^\infty e^{\pm iy \frac{x}{2t}}  e^{\frac{i|y|^2}{4t}} f(y) dy
	=[MD\mathcal{F}^{\pm}Mf](x).
\end{align*}
Therefore, it holds from \eqref{eqA.1}, this factorization formula, and the definition of the Fourier transform $\mathcal{F}$ that $U(t)=e^{it\Delta_K} = \mathcal{M} \mathcal{D} \mathcal{F} \mathcal{M}$.
\end{proof}

We will show Lemma \ref{lem4.2}, which follows from the following lemmas. 

\begin{lemma}
\label{lemA.3}
For $\varphi \in H^1(\mathbb{R}_{+})$, 
we have
\begin{align*}
	x\mathcal{F}^{\pm}\varphi=\frac{\pm i}{\sqrt{2\pi}}  \varphi(0)
	\pm i \mathcal{F}^{\pm} (\partial_{x}\varphi)
\end{align*}
and equivalently
\begin{align*}
	\mathcal{F}^{\pm}(\partial_{x}\varphi)
	= \mp i x\mathcal{F}^{\pm}\varphi -\frac{1}{\sqrt{2\pi}}  \varphi(0).
\end{align*}
\end{lemma}

\begin{proof}
By integration by parts, we see that
\begin{align*}
	x\mathcal{F}^{\pm}\varphi
	&=\frac{\mp i}{\sqrt{2\pi}} \int_{0}^{\infty} \partial_{y} (e^{\pm ixy }) \varphi(y) dy
	\\
	&=\frac{\mp i}{\sqrt{2\pi}} \left[  e^{\pm ixy} \varphi(y) \right]_{y=0}^{y=\infty}
	-\frac{\mp i}{\sqrt{2\pi}}  \int_{0}^{\infty} e^{\pm ixy}  (\partial_{x} \varphi)(y) dy
	\\
	&=\frac{\pm i}{\sqrt{2\pi}}  \varphi(0)
	\pm i \mathcal{F}^{\pm} (\partial_{x}\varphi).
\end{align*}
This completes the proof. 
\end{proof}

\begin{proof}[Proof of Lemma \ref{lem4.2}]
For $\varphi \in H_{c}^{1}(\mathcal{G})$, it holds from Lemma \ref{lemA.3} that
\begin{align*}
	(X\mathcal{F}^{-1} \varphi)_{j}
	&=\left(X\left\{(\mathcal{F}^{+}-\mathcal{F}^{-})I_n
	+ \frac{2}{n} \mathcal{F}^{-}J_n \right\} \varphi \right)_{j}
	\\
	&=x (\mathcal{F}^{+}-\mathcal{F}^{-})\varphi_j
	+ \frac{2}{n} x\mathcal{F}^{-} \left( \sum_{k=1}^{n} \varphi_k \right)
	\\
	&= \frac{i}{\sqrt{2\pi}} \varphi_j(0) +i \mathcal{F}^{+}(\partial_x \varphi_{j})
	\\
	&\quad - \left( \frac{-i}{\sqrt{2\pi}}  \varphi_{j}(0)
	- i \mathcal{F}^{-} (\partial_{x}\varphi_{j})\right)
	\\
	& \quad +\frac{2}{n} \left\{ \frac{-i}{\sqrt{2\pi}} \sum_{k=1}^{n} \varphi_k(0) - i \mathcal{F^{-}} \sum_{k=1}^{n} \partial_x \varphi_k \right\}
	\\
	&=i \left\{( \mathcal{F}^{+} +  \mathcal{F}^{-})\partial_x \varphi_{j} - \frac2{n} \mathcal{F}^{-} \sum_{k=1}^{n} \partial_x \varphi_k\right\}
	\\
	&\quad +\frac{2 i}{\sqrt{2\pi}} \varphi_j(0) -\frac{2}{n} \frac{i}{\sqrt{2\pi}} \sum_{k=1}^{n} \varphi_k(0).
\end{align*}
By the continuity at the vertex of $\varphi$, the summation of the last two terms is zero.  Thus, we obtain the desired equality. 
\end{proof}

Next, we will show Lemma \ref{lem4.4}. We do not require the continuity of $\varphi$ at the origin in Lemma \ref{lem4.4} and the following lemma. 

\begin{lemma}
\label{lemA.4}
For $\varphi \in H^{0,1}(\mathbb{R}_{+})$, we have
\begin{align*}
	\partial_{x} \mathcal{F}^{\pm}\varphi
	=\pm i \mathcal{F}^{\pm} (y\varphi)
\end{align*}
\end{lemma}

\begin{proof}
We have
\begin{align*}
	\partial_{x}\mathcal{F}^{\pm}\varphi
	=\frac{1}{\sqrt{2\pi}} \int_{0}^{\infty} (\pm i y) e^{\pm ixy} \varphi(y) dy
	= \pm i \mathcal{F}^{\pm} (y\varphi).
\end{align*}
\end{proof}

\begin{proof}[Proof of Lemma \ref{lem4.4}]
It follows from Lemma \ref{lemA.4} that
\begin{align*}
	\partial_x \mathcal{F} 
	&= \partial_x (\mathcal{F}^{-}-\mathcal{F}^{+})I_{n}
	+ \frac{2}{n} \partial_x \mathcal{F}^{+}J_{n}
	\\
	&=i\left\{ (-\mathcal{F}^{-}-\mathcal{F}^{+})X I_{n}
	+ \frac{2}{n}\mathcal{F}^{+} X J_{n}\right\}
	\\
	&=-i \mathcal{F}_{c} X.
\end{align*}
This completes the proof of Lemma \ref{lem4.4}. 
\end{proof}

It follows from \cite{Wed15} that the Fourier transform and co-Fourier transform are unitary on $L^2(\mathcal{G})$. 

\begin{proof}[Proof of Lemma \ref{lem4.1}]
The result for $\mathcal{F}$ in Lemma \ref{lem4.1} follows from \cite{Wed15}. 
We note that $\mathcal{F}_{c}$ is the Fourier transform with respect to the Laplacian $\Delta_M$ defined as follows:
\begin{align*}
	\mathscr{D}(\Delta_M) &:= \{f \in H^2(\cG)\ :\ Bf(0) + Af'(0+)=0\}, 
	\\
	\Delta_M f &:= (f_1'', f_2'', \cdots, f_n''),
\end{align*}
where $A$ and $B$ are in \eqref{eq2.7}. Thus, the general theory by \cite{Wed15} implies the desired statement. 
\end{proof}

The Hausdorff--Young inequality follows immediately as follows.
\begin{proof}[Proof of Lemma \ref{HY}]
As seen above, $\mathcal{F}$ is unitary in $L^2(\mathcal{G})$. Since
\begin{align*}
	\| \mathcal{F}^{\pm} f \|_{L^\infty(0,\infty)} \leq \| f \|_{L^1(0,\infty)},
\end{align*}
we have $\| \mathcal{F}f\|_{L^\infty(\mathcal{G})} \cleq \|f\|_{L^1(\mathcal{G})}$. Therefore, interpolation implies the desired estimate. Similar estimate for $\mathcal{F}_{c}$ also holds. 
\end{proof}

We show Lemma \ref{lem4.5}. 

\begin{proof}[Proof of Lemma \ref{lem4.5}]
By the H\"{o}lder continuity of $e^{ix}$, we have $|e^{\frac{i|x|^2}{4t}}-1| \cleq |t|^{-\alpha}|x|^{2\alpha}$ for $0\leq \alpha\leq 1/2$. 
Thus, it holds that
\begin{align*}
	\norm{(\mathcal{M}-1)f}_{L^p(\mathcal{G})} 
	&=\sum_{j=1}^{n} \norm{(M-1)f_j}_{L^p(0,\infty)}
	\\
	&\cleq  |t|^{-\alpha}\sum_{j=1}^{n} \norm{|x|^{2\alpha}f_j}_{L^p(0,\infty)}
	= |t|^{-\alpha} \norm{X^{2\alpha}f}_{L^p(\mathcal{G})}.
\end{align*}
This completes the proof. 
\end{proof}

\begin{proof}[Proof of Lemma \ref{iv:0}]
The first inequality follows from 
\begin{align*}
	\norm{XU(-t)|f|^2f}_{L^2} 
	&=\norm{ X \mathcal{F}^{-1} \mathcal{D}^{-1} \mathcal{M}^{-1}|f|^2f}_{L^2} 	\\
	&=\norm{\mathcal{F}_{c}^{-1} \partial_{x} \left( \mathcal{D}^{-1} \mathcal{M}^{-1} |f|^2f \right)}_{L^2} 	\\
	&\leq C\norm{\mathcal{D}^{-1} t \partial_{x} \left(  |\mathcal{M}^{-1} f|^2 \mathcal{M}^{-1}f \right)}_{L^2} 	\\
	&\leq C\norm{ \mathcal{M}^{-1}f}_{L^\infty}^2 \norm{t\partial_x ( \mathcal{M}^{-1}f)}_{L^2}
	\\
	&\leq c_0 \norm{f}_{L^\infty}^2 \norm{XU(-t)f}_{L^2}.
\end{align*} 
The second inequality holds by the Leibniz rule. 
\end{proof}

At last, we give the proof of the Sobolev inequality. 

\begin{proof}[Proof of Proposition \ref{sob:1}]
Let $f$ be a function on the star graph.
Obviously we have the following estimate. For $x \geq 0$, 
\begin{align*}
	|f_j(x)|^2
	&= -\int_{x}^{\infty} (|f_j(y)|^2)' dy
	\\
	&\leq  2 \int_{0}^{\infty} |f_j(y)|| f_j'(y)| dy
	\\
	&\leq 2 \| f_j \|_{L^2(0,\infty)}  \| f_j' \|_{L^2(0,\infty)} 
	\\
	&\leq 2 \| f \|_{H^1(\mathcal{G})}^2.
\end{align*}
Thus, we get $\| f_j \|_{L^\infty(0,\infty)}\leq 2 \| f \|_{H^1(\mathcal{G})}^2$.
Therefore, 
\begin{align*}
	\| f \|_{L^{\infty} (\mathcal{G})} = \max_{j=1,\cdots,n} \| f_j \|_{L^\infty(0,\infty)} \leq \sqrt{2} \| f \|_{H^1(\mathcal{G})}.
\end{align*}
For $2 \leq p < \infty$, we obtain
\begin{align*}
	\| f \|_{L^{p} (\mathcal{G})}^p
	&\leq  \sum_{j=1}^{n} \| f_j \|_{L^{2}  (0,\infty)}^{2} \| f_j \|_{L^{\infty}  (0,\infty)}^{p-2}
	\\
	&\leq 2^{\frac{p-2}{2}} \| f \|_{H^1(\mathcal{G})}^{p-2} \| f \|_{L^{2} (\mathcal{G})}^2
	\leq 2^{\frac{p-2}{2}} \| f \|_{H^1(\mathcal{G})}^{p}
\end{align*}
Thus, we obtain the desired estimate. 
\end{proof}


\end{document}